\documentclass[11pt]{amsart}

\usepackage{amsmath,amsfonts,amssymb,amscd,amsthm,amsbsy,upref,color}

\usepackage{lineno}   

\usepackage{subfigure}

\usepackage{eucal}

\usepackage{hyperref}

\usepackage{esint}

\usepackage{mathrsfs}

\usepackage{enumitem}
\usepackage{graphics}
\usepackage{graphicx}

\usepackage{geometry}
\hypersetup{
    bookmarks=true,         
    unicode=true,          
    pdftoolbar=true,        
    pdfmenubar=true,        
    pdffitwindow=false,     
    pdfstartview={FitH},    
    pdftitle={My title},    
    pdfauthor={Author},     
    pdfsubject={Subject},   
    pdfcreator={Creator},   
    pdfproducer={Producer}, 
    pdfkeywords={keyword1} {key2} {key3}, 
    pdfnewwindow=true,      
    colorlinks=true,       
    linkcolor=blue,          
    citecolor=blue,        
    filecolor=magenta,      
    urlcolor=cyan           
}

\renewenvironment{proof}[1][\proofname ]{{\noindent \bfseries #1. }}{\qed \bigskip } 

\definecolor{orange}{RGB}{220,110,0}
\definecolor{violet}{RGB}{141,10,100}

\newcommand{\R}{{\mathbb R}}

\newcommand{\e}{\varepsilon}

\newcommand\p{\partial}
\newcommand\fb[1]{\Gamma_{#1}}

\newcommand\graph[1]{{\rm{Graph}}_{#1}}

\newcommand\po[1]{\{#1>0\}}

\def\Om{\Omega}
\def\na{\nabla}
\def\W{\mathbb W}

\newtheorem{theorem}{Theorem}[section]

\newtheorem{cor}[theorem]{Corollary}

\newtheorem{defn}{Definition}[section]
\newtheorem{example}[theorem]{Example}

\newtheorem{lem}[theorem]{Lemma}

\newtheorem{prop}[theorem]{Proposition}
\newtheorem{remark}[theorem]{Remark}

\newtheorem{sublem}[theorem]{Sublemma}

\newtheoremstyle{named}{}{}{\itshape}{}{\bfseries}{.}{.5em}{\thmnote{#3 }#1}
\theoremstyle{named}

\newtheorem*{AAA}{Theorem A}

\newtheorem*{BBB}{Theorem B}

\numberwithin{equation}{section}

\makeatletter
\def\@tocline#1#2#3#4#5#6#7{\relax
  \ifnum #1>\c@tocdepth 
  \else
    \par \addpenalty\@secpenalty\addvspace{#2}%
    \begingroup \hyphenpenalty\@M
    \@ifempty{#4}{%
      \@tempdima\csname r@tocindent\number#1\endcsname\relax
    }{%
      \@tempdima#4\relax
    }%
    \parindent\z@ \leftskip#3\relax \advance\leftskip\@tempdima\relax
    \rightskip\@pnumwidth plus4em \parfillskip-\@pnumwidth
    #5\leavevmode\hskip-\@tempdima
      \ifcase #1
       \or\or \hskip 2em \or \hskip 2em \else \hskip 3em \fi%
      #6\nobreak\relax
    \dotfill\hbox to\@pnumwidth{\@tocpagenum{#7}}\par
    \nobreak
    \endgroup
  \fi}
\makeatother

\title[{$K$}-surfaces with free boundaries]{$K$-surfaces with free boundaries}

\author{Hayk Aleksanyan}
\address{Department of Mathematics, KTH Royal Institute of Technology, SE-100 44  Stockholm,
Sweden}
\email{hayk.aleksanyan@gmail.com}

\author{Aram L. Karakhanyan}
\address{Maxwell Institute for Mathematical Sciences and School of Mathematics,
University of Edinburgh, James Clerk Maxwell Building, Peter Guthrie Tait Road,
Edinburgh EH9 3FD, United Kingdom}
\email{aram6k@gmail.com}

\thanks{H. A. was supported by postdoctoral fellowship from Knut and Alice Wallenberg Foundation. }

\begin{document}
\baselineskip=15pt    

\begin{abstract}
A well-known question in classical differential geometry and geometric analysis
asks for a description of possible boundaries of $ K$-surfaces,
which are smooth, compact hypersurfaces in $\R^d$ having constant
Gauss curvature equal to $K \geq 0$. This question
generated a considerable amount of remarkable results in the last few decades.
Motivated by these developments here we study the question
of determining a $K$-surface when only part of its boundary is
fixed, and in addition the surface hits a given manifold at some fixed angle. 
While this general setting is out of reach for us at the present,
we settle a model case of the problem, which in its analytic formulation 
reduces to  a Bernoulli type free boundary problem for the Monge-Amp\`ere equation.
We study both the cases of 0-curvature and of positive curvature.
The formulation of the free boundary condition and its regularity are the most delicate and challenging questions addressed in this work.  
In this regard we introduce a notion of a \textit{Blaschke extension} of a solution
which might be of independent interest.

The problem we study can also be interpreted as the  Alt-Caffarelli problem for the 
Monge-Amp\`ere equation. Moreover, it also relates to the problem of isometric embedding of a 
positive metric on the annulus with partially prescribed boundary and optimal transport with free mass.

\bigskip

\noindent 

\vspace{0.1cm}


\vspace{0.1cm}


\end{abstract}


\maketitle

{\small{ \tableofcontents}}

\section{Introduction}

\subsection{Background and motivation}
$K$-surface in $\R^{d+1}$ $(d\geq 2)$ is a smooth compact hypersurface of
constant Gauss curvature $K$. A problem of fundamental importance in classical differential geometry and geometric analysis
concerns description of possible boundaries of $K$-surfaces.
Precisely, when a collection $\gamma=\{\gamma_1, ...,\gamma_n \}$ of disjoint $(d-1)$-dimensional
closed smooth embedded submanifolds can form a boundary of a (in general immersed) $K$-surface\footnote{
Note, that a $K$-surface without a boundary with $K>0$, is a boundary of a convex body.} in $\R^{d+1}$?
In particular, S.-T. Yau \cite[Problem 26]{Yau} in his famous list of open problems asks for conditions on a space curve $\gamma \subset \R^3$
to be the image of the boundary of an isometric embedding of the given smooth metric with positive curvature on the disk.
In their analytic formulation, locally these problems mostly reduce to 
equations of Monge-Amp\`ere type. Moreover, for the case of positive curvature $K>0$,
the analysis is restricted to the case of elliptic equations. This drives the study to the
class of \emph{locally strictly convex} hypersurfaces\footnote{These surfaces locally lie on one side of their tangent hyperplanes,
but need not do so globally.}.
With this in mind, an immediate necessary condition on $\gamma$ to bound a locally strictly convex hypersurface, 
is that its second fundamental form must be everywhere non-degenerate
(in particular, in $\R^3$ this means that the curve $\gamma$ is free of inflection points).
This elementary limitation, however, is not sufficient. It was shown by Rosenberg \cite{Rosenberg}
that there are further topological objections on $\gamma$, which was later demonstrated by Gluck and Pan \cite{Gluck-Pan} to be not sufficient either. Most recently, Ghomi \cite{Ghomi17}, answering a question of Rosenberg
\cite{Rosenberg} from 1993, proved that the torsion of any closed space curve
bounding a simply connected locally convex surface vanishes at least 4 times.
This important result is the only new necessary condition for the existence
of $K$-surfaces since the condition obtained by Rosenberg in early 90's involving the self-linking number of the curve. 
We give a short review of the literature in subsection \ref{subsec-related} below.

For many years these problems have served as an important bridge between differential geometry and
the analysis of partial differential equations, in particular for equations of Monge-Amp\`ere type,
where the advance on either side motivated and lead to new discoveries on the other end.
Inspired by these developments, in this paper we introduce and study a class of $K$-surfaces
with free boundaries. Namely, instead of prescribing 
a set $\gamma$ of closed smooth strictly convex codimension 2 submanifolds, and asking
for a (locally) convex smooth manifold spanning $\gamma$,
we treat the given $\gamma$ as part of the boundary of the sought-for surface.
Then, fixing a smooth embedded submanifold $\gamma$ of codimension 1,
we ask if there is a $K$-surface spanning $\gamma$ 
and hitting the (given) target manifold $T_0$ at some prescribed constant angle.
In this generality, the problem seems to be out of reach,
and here we will study a special case when the boundary of the surface lies in a hyperplane,
and the target manifold is a hyperplane parallel to the one containing the portion
of the fixed boundary. This basic geometric setting already presents a non-trivial challenge, and seems to be an important model case treated for instance in \cite{HRS}
for the case of classical $K$-surfaces discussed above.
We carry out the analysis in two distinct cases, namely when the curvature is vanishing,
and when the curvature is strictly positive. 
We examine qualitative properties of the problem,
i.e. existence, uniqueness and regularity of solutions, as well as regularity and geometric properties of
associated free boundaries. Our methods are of both geometric and analytic nature, 
drawing ideas from convex geometry, theory of Monge-Amp\`ere equations, and free boundary problems.

\subsection{Formal setting and main results}
We now proceed to formal definitions. 
Fix $\gamma=\{\gamma_1, ...,\gamma_n \}$ a collection of disjoint $(d-1)$-dimensional
closed smooth embedded submanifolds in $\R^{d+1}$, and let
$T_0$ be a smooth embedded manifold of codimension 1 in $\R^{d+1}$, and $\lambda_0>0$ be a fixed constant.
The general question, we have in mind, which we refer to as \emph{$K$-surfaces with free boundaries},
is the following:

\smallskip
\emph{When does there exist a $K$-surface which spans $\gamma$ and hits $T_0$ at the prescribed constant angle $\theta > 0$ ?}

\smallskip

This is of course a generalization of the classical question on $K$-surfaces to allow
free boundaries. A model of this problem, which we will study here, is when the target manifold $T_0 = \R^d \times \{0\}$ 
and $\gamma \subset \R^d \times \{h_0\}$ is the boundary of some bounded convex set, and $h_0>0$ is a given constant.
Then, one asks for the existence of a $K$-surface, having $\gamma$ on its boundary
and intersecting $T_0$ at a prescribed constant angle $\theta$. In analytic terms,
we ask for solvability of the following Bernoulli type free boundary  problem
for the Monge-Amp\`ere equation: given a convex domain $\Omega \subset \R^d\times \{0\}$ and parameters $h_0, \lambda_0>0$,
$K_0 \geq 0$
find a concave function $u:\R^d\times \{0\} \to \R_+$ such that
\begin{equation}\label{prob-main}
\begin{cases}  \mathrm{det} D^2 (- u)=K_0 \psi(|\na u|), &\text{in $\left( \R^d \setminus \overline{\Omega} \right)$}\cap \{u>0\}, \\ 
 u=h_0 ,&\text{on $ \partial \Omega    $}, \\
 |\nabla u| =\lambda_0, &\text{on $\Gamma_u$},   \end{cases}
\end{equation}
where $\psi > 0$ is a prescribed real-valued $C^{\infty}$ function, $\Gamma_u = \partial \{ u>0\} \setminus \overline{\Omega}$
is the free boundary, and the gradient condition is clarified below.
The gradient value $\lambda_0$ corresponds to the hitting angle $\theta  = \arccos \frac1{\sqrt{1+\lambda_0^2}}$,
as will be clear later.
The three main choices for $\psi$, which are of geometric interest, include $\psi(\xi)=0$,  $\psi(\xi)=1$, and $\psi(\xi)=(1+|\xi|^2)^{\frac{d+2}2}$,
corresponding to the cases of \emph{zero curvature}, \emph{positive curvature measure}, and \emph{positive constant Gauss curvature}, respectively.
The latter case for $\psi$ defines the $K$-hypersurface.

We remark that a problem \eqref{prob-main} for $p$-Laplace operator
was studied by Henrot and Shahgholian \cite{H-Sh}. Although the analytic formulation
of the problem is similar here, our motivation and methods are entirely different from those of \cite{H-Sh}.

\subsubsection{The free boundary condition}\label{subsec-FB} 
The gradient condition on boundary in \eqref{prob-main}
is understood in a weak sense. Namely, whenever at $x_0\in \Gamma_u$ 
there is a well-defined unit inner normal to the set $\Gamma_u$, then the normal derivative of $u$ at $x_0$ equals $\lambda_0$.
A few remarks are in order. First of all, at such $x_0$ the normal derivative of $u$
exists (possibly having infinite value). This is in view of a standard fact that
concave (convex) functions have directional derivatives, which we briefly recall here.
Indeed, fix any $x_0 \in \Gamma_u$  and let $\nu$ be the unit inward normal of $\Gamma_u$
at $x_0$. Then, for any small parameters $t_2>t_1>0$ we have
$$
u(x_0  + t_1 \nu ) = u \left(  \left(1- \frac{t_1}{t_2} \right) x_0 + \frac{t_1}{t_2} (x_0 + t_2 \nu)   \right) \geq
\frac{t_1}{t_2} u(x_0 + t_2 \nu),
$$
where in the last inequality we have used that $u$ is concave and $u(x_0)=0$.
We get that the function $t\mapsto \frac{u(x_0 + t\nu)}{t}$ is decreasing in $t>0$
and hence the existence of the normal derivative at $x_0$.
Finally, we observe that the convexity of the set $\Gamma_u$ implies that it has well-defined
normal almost everywhere (with respect to the surface measure), and hence
in \eqref{prob-main} the free boundary condition must hold almost everywhere.

It will be convenient for us, to work with a geometric reformulation of the
last requirement of \eqref{prob-main}, in terms of the \emph{slope} of a support hyperplane,
to which we now proceed.
Consider the convex body $\mathcal{K}$ bounded by the graph of $u$ and hyperplanes $\R^d \times \{0\}$
and $\R^d \times \{h_0\}$. Let also $X_0  = (x_0, 0) \in \R^d\times \R$ be a point
on $\Gamma_u$ where normal to the free boundary exists. Now let $H = \{ X\in \R^{d+1}: \ (X - X_0) \cdot \nu =0 \}$
be a support hyperplane to $\mathcal{K}$ at $X_0$, where $\nu= (\nu_1,...,\nu_{d+1}) \in \R^{d+1}$ is a unit vector.
Clearly the hyperplane $G: = H\cap ( \R^d\times \{0\} ) $
defines a support plane for $\Gamma_u$ through $x_0$, and due to the existence of a
normal at $x_0$, $G$ is the unique support hyperplane to $\Gamma_u$ through $x_0$. We conclude that $H$ has only one degree of freedom,
determined by the angle it makes with $\R^d\times \{0\}$.
From now on, we will only consider the set of those $H$ which are above $\Omega$, i.e. the last coordinate of their unit normal is negative.
Using the fact $|\nu_{d+1}|<1 $ we define $\nu_* : = \frac{1}{ ( 1-\nu_{d+1}^2 )^{1/2} } \overline{\nu}$
where $\overline{\nu} \in \R^d$ is formed from the first $d$ coordinates of $\nu$.
By definition $|\nu_*|=1$, and by construction the linear function $H(x) : =  -\frac{\overline{\nu}}{\nu_{d+1}} \cdot (x-x_0) $
bounds $u$ from above, and in particular $\nu_{d+1}<0$.
Note that if $\nu$ is the unit normal to the support function at a free boundary point $x_0$
with $\nu_{d+1}<0$ then $ \nu_*$ is the inner unit normal at $x_0$.
We now fix the following:
\begin{defn}
For any hyperplane $H = \{X \in \R^{d+1}: \ (X-X_0)\cdot \nu =0 \}$, where $\nu=(\nu_1,...,\nu_{d+1})$
is a unit vector, the vector $-\frac{\overline{\nu}}{\nu_{d+1}}$ is called the slope of $H$.
\end{defn}
When it will be clear from the context, we will abuse the notation, and refer to the length of
the slope, as the slope of a hyperplane.
Getting back to the last condition in \eqref{prob-main}, we conclude that at any point $x_0\in \Gamma_u$, having a well-defined unit inward normal $\nu$,
the equivalent definition of the free boundary condition \eqref{prob-main}
is that among all support planes to the graph of $u$ at $x_0$ and staying above $\Omega$,
the smallest slope has length $\lambda_0$. This is easy to check, using the discussion above.
Geometrically, it means that the hyperplane $H$ cannot be inclined on the graph of $u$ more 
than the angle $\arccos (1+\lambda_0^2)^{-1/2}$. Such support functions will be referred to as {\bf extreme}.

\subsubsection{Weak solutions}\label{subsec-weakSol}
Here we define the notion of a weak solution (after A.D. Aleksandrov) to \eqref{prob-main} and recall some key concepts related to the 
Monge-Amp\`ere operator.

\begin{defn}
Let $u$ be a convex function defined in some domain $D \subset \R^d$. 
Given any $x_0\in D$, the following set of slopes
$$
\omega_{x_0}(u) =\{p\in \R^d : \  \ u(x) \ge p\cdot (x-x_0)+u(x_0)\quad  \forall x\in D \},
$$
is called the gradient mapping of $u$ at $x_0$; correspondingly for
a set $E\subset \Omega$ we set
$$
\omega_E (u)=\bigcup_{x\in E}\omega_x (u).
$$
\end{defn}

Due to Aleksandrov's theorem $\omega_E$ is measurable for any Borel set $E\subset D$
and induces a Borel measure on $D$ by the following way
\begin{equation}\label{MA-meas}
\mu(E)=|\omega_E(u)|,
\end{equation}
where $E\subset D$ is a Borel set, and $|\cdot | $ stands for the Lebesgue measure.
The measure $\mu$ defined by \eqref{MA-meas} is called the Monge-Amp\`ere measure.
In particular, one has the following fundamental property:
the set of all slopes $p\in \R^d$ which belong to the gradient image of more than one point of $D$
has Lebesgue measure 0 (see \cite[p. 190]{Al55}, also
\cite[Theorem 2.5]{Ra-Ta}, and \cite{Gut-b}).

An important property of the Monge-Amp\`ere measure is its weak continuity in the following sense: if 
$\{v_m\}$ is a sequence of convex functions on $D$ such that 
$v_m\to v$ locally uniformly in $D$ then 
the measure $\mu_m(E)=|\chi_{v_m}(E)|$ converges weakly to 
$\mu(E)=|\chi_v(E)|$, see  \cite{Urbas}, \cite{Gut-b}.
This is also true for the measure 
\[\mu_\psi(E)=\int_{\omega_E(u)}\frac{d\xi}{\psi(\xi)}\]
see \cite{Pog-MA}.

We next define the solution to \eqref{prob-main}.

\begin{defn}\label{defn-Aleks}
We say that a concave function $u:\R^d \setminus \Om \to \R$ is a weak solution 
(corr. super-solution) to \eqref{prob-main}
if the following hold:
\begin{itemize}
\item[\normalfont{(a)}] for any Borel set $E\subset \{u>0\}$ we have
$$
\int_{\omega_E(-u)  } \frac{d \xi}{\psi(  \xi  ) }  \ (\geq) \ = K_0 |E|,
$$
where $\omega_E(-u)$ is the gradient image of $-u$  on $E$,
\item[\normalfont{(b)}] $u=h_0$ on $\partial \Omega$ pointwise,
\item[\normalfont{(c)}] for any point $x_0 \in  \p\po u \setminus \overline{\Omega} =: \Gamma_u $ on the free boundary where the set $\Gamma_u$ has a well-defined normal,
and among the all support hyperplanes to the graph of $u$ at $x_0$
the smallest slope equals (is bounded above by) $\lambda_0$.
The class of super-solutions is denoted by $\W_+(K_0, \lambda_0, \Om)$.
\end{itemize}
\end{defn}

The last condition (c) with slopes should be understood in a sense discussed in subsection \ref{subsec-FB}.
Also, when the constant $K_0=0$ in \eqref{prob-main}, the condition (a) above, means that the Monge-Amp\`ere measure
associated with $u$ is trivial.

\subsubsection{Main results}
The following are our main results.

\begin{AAA}
In \eqref{prob-main} take $K_0=0$ and assume $\Om$ is a bounded convex domain,
which is $C^{1,1}$-regular. Then, there exists a unique concave function $u$ which is a weak solution \eqref{prob-main}. 
Moreover, the solution $u$ is a ruled surface, which is $C^{1,1}$ in 
the set $\{u>0\}\setminus \overline{\Omega}$, and the free boundary $\fb u $ is $C^{1,1}$ also.
If in addition $\Om$ is strictly convex, then so is the free boundary.
\end{AAA}

\begin{remark}
The methods of our paper allow for the following (easy) generalization.
Instead of the hyperplane $\R^d\times \{h_0\}$, where the boundary of the surface was prescribed,
we can consider a hyperplane of the form
$$
\Sigma= \{ x\in \R^{d+1}: \ x_{d+1} = l_1 x_1 + ... + l_d x_d + a \}, \ \ l=(l_1,...,l_d)\in \R^d,
$$
and then let $\gamma-$the boundary of the surface (i.e. the boundary data of $u$ in \eqref{prob-main}),
to be equal to the intersection of the cylinder $\partial \Omega \times \R $  and $\Sigma$.
Here we should only require that $\Omega$ and $\Sigma$ are chosen so that $\gamma$
does not intersect the target hyperplane $T_0=\R^{d}\times \{0\}$.
Next, assuming the compatibility condition that $\sum\limits_{i=1}^d l_i^2 \leq \lambda_0^2$, i.e.
the target hyperplane can be seen from $\Sigma$ at a contact angle at least $\arccos (1+\lambda_0)^{-1/2}$,
we can still work out the details of our arguments.
\end{remark}

\begin{BBB}
Let $\Om$ be uniformly strictly convex and $C^2$ regular. Let 
$\psi(\xi)=\psi(|\xi|)$ be $C^\infty$ smooth nondecreasing positive function. Then, there is a universal constant 
$K = K(\lambda_0, \Omega, \psi)>0$ small such that if $K_0 \in (0, K)$ then \eqref{prob-main}
has a weak solution $u$, which is $C^\infty$ in $\{ u>0 \} \setminus \overline{\Omega}$ and the free boundary $\fb u$ is 
$C^{\infty}$ regular. 
\end{BBB}

The smallness assumption on $K_0$ cannot be eliminated entirely as indicated by explicit computations for the radially symmetric solutions, see the Appendix.

\smallskip

We prove Theorem A in Section \ref{sec-homogen}, and Theorem B in Section \ref{sec-nonZero}.

\begin{remark}
When $\psi(t) =1$, in \eqref{prob-main} we get the equation for prescribed curvature measure,
and for $\psi(t) = (1+t^2)^{(d+2)/2}$ we recover the case of prescribed Gauss curvature.
In both cases condition \eqref{ineq-star}, which reads
$$
 K_0^{\frac 1d} \leq \psi^{-\frac 1d} (\lambda_0) \frac{  \kappa_0 \lambda_0^2 }{ h_0 \kappa_0 +  ( \lambda_0^2 + h_0^2 \kappa_0^2 )^{1/2} },
$$
where $\kappa_0$ is the smallest curvature of $\p \Om$,
is sufficient for the existence of a solution, and it is in the regularity
of the free boundary that we need to make $K_0 $ even smaller.
We also remark that for $\psi=1$ the case of equality in the above inequality
coincides, as one would expect, with the identity in \eqref{res1} concerning the radially symmetric solutions.

Computations similar to those we did in subsection \ref{subsec-comp}
show that for general parameters involved \eqref{prob-main} there can be no solution
for the prescribed Gauss curvature equation. The computations, while tedious,
present no difficulty, and we will leave it out for the interested reader to explore.

We also remark that the monotonicity of $\psi$ is being used to get a neat
formulation for the free boundary condition. One should be able
to treat more general cases for $\psi$, but here we do  not attempt to full generality, and rather
focus on model of the problem.
\end{remark}

\smallskip

\noindent \textbf{Notation.}
The following notation will be in force throughout the paper.

\smallskip

\begin{tabbing}
$C, C_0, C_n, ... $ \hspace{1.55cm}       \=\hbox{positive constants, that can vary from formula to formula}\\
$\partial U$        \>\hbox{the boundary of a set} $U$,\\
$\Sigma$        \>\hbox{the hyperplane } $\R^d\times\{h_0\} \subset \R^{d+1} $, \\
$\widehat\Omega$\> $({\Omega}\times\R) \cap \Sigma$,\\
$B(x, r)$ \>\hbox{the $d$-dimensional Euclidean ball of radius $r>0$ and center at $x$}, \\
$\Gamma_u= \{u>0\}\setminus \overline{\Om}$      \> \hbox{the  free boundary},\\
$\mathrm{Hull}(E)$      \> \hbox{the convex hull of a set $E\subset \R^d$},\\
$\nu$\> \hbox{the unit inner normal}, \\
$T_0$\> \hbox{the target manifold $\R^d\times\{0\}$,}\\
$\W_+(K_0, \lambda_0, d)$\> \hbox{the class of super-solutions, see Definition \ref{defn-Aleks}}.\\

\end{tabbing}

\subsection{Related works}\label{subsec-related}
We present a short review of the theory of $K$-surfaces. The literature quoted below is by no means exhaustive, 
but is selected, slightly arbitrarily, with a hope to give the reader a flavour of the breadth and scopes of this 
area of research.

One may roughly categorise the results on existence of $K$-surfaces in two groups;
namely, those where the sought-for $K$-surface is globally a graph of a function over some (nice) domain,
and those where the surface is not necessarily a graph globally (this setting allows immersed surfaces too, in particular).
The main advantage of the first scenario, is the availability of a global coordinate system,
where the existence of the surface in question, can be reformulated in terms of some boundary value problem
for Monge-Amp\`ere equation. Then, having some nice control over the geometry of the domain where the problem is posed,
the original geometric problem is being treated as a question in the existence theory for Monge-Amp\`ere equations.
In this direction, the results of Caffarelli, Nirenberg, and Spruck \cite{CNS-CPAM}, show that if $\Omega\subset \R^d$ is a bounded, smooth and a strictly convex domain,
then for any smooth function $\varphi$ over $\partial \Omega$, if $K>0$ is small enough, then there is a $K$-surface spanning the graph of $\varphi$.
This existence result was extended by Hoffman, Rosenberg, and Spruck \cite{HRS} to the case of boundaries with multiple components,
where among other results, it was proved that given any two strictly convex curves lying in two parallel hyperplanes,
and such that one of the curves can be moved inside the other by a translation,
then for a sufficiently small $K>0$ there exists a $K$-surface having these curves as its boundary.
Both of these papers treat the existence of $K$-surfaces primarily through the existence theory of Monge-Amp\`ere equations.
Later on, more general settings were treated by Guan and Spruck \cite{Guan-Spruck}, and Guan \cite{Guan}.
One of the results from \cite{Guan-Spruck}, which is particularly striking, proves the existence of locally convex
embedded $K$-surfaces with arbitrary high genus. This remarkable result shows that the class of $K$-surfaces can be very complicated in spite of 
strict conditions on its geometry. 

We now move to the second category of the existence results, where the surface is not necessarily a graph.
In this direction, the paper of Guan and Spruck \cite{Guan-Spruck-JDG}, establishes a rather
general existence result. Namely, assuming the existence of a locally convex immersed and $C^2$-smooth hypersurface
$\Sigma \subset \R^{d+1}$, which is locally strictly convex along its boundary, it is proved
that for any positive $K$ which nowhere exceeds the curvature of $\Sigma$, there exists an up to the boundary smooth,
locally strictly convex immersed hypersurface $M$ having $\partial \Sigma$ as its boundary and having curvature $K$
everywhere.
This result was also independently obtained by Trudinger and Wang \cite{TW}.
Finally, let us quote the important paper by Ghomi \cite{Ghomi01} which brings direct geometrical and topological ideas
into the existence theory of $K$-surfaces.

In the last part of our review we note that the all cases discussed above,
treated the case of strictly positive curvature.
For the vanishing curvature case, which comprises another important class of problems,
it is proved by Guan and Spruck \cite{Guan-Spruck-JDG} that if $\gamma$
bounds a locally strictly convex codimension 2 hypersurface, which is $C^2$-smooth, there exists a locally convex $C^{1,1}$-smooth hypersurface $M$ having $\gamma$ as its boundary.

\subsection{Optimal  transport with free boundary between annuli}

The problem \eqref{prob-main} can be regarded as the analogue of Alt-Caffarelli problem 
\cite{AC} for the Monge-Amp\`ere equation. 
 
Another interesting  interpretation can be given in terms of optimal mass transport theory. Suppose $\psi=1$, $0\in\Om$ and let $w=u+\frac12 |x|^2$. Then 
$-w$ is a potential (modulo a constant summand) defining the transport map $y=x-\na w$ which maps the annulus bounded by $\fb u$ and $\p \Om$
to the one bounded by $\p B(0, \lambda_0)$ and $\na u(\fb u)$, see Figure \ref{Fig-ring}.

Moreover, $w$ solves the following overdetermined problem for the unknown $w$: 
\begin{equation}\label{MA-AC}
\begin{cases}  
\mathrm{det} (-D^2 w+\hbox{Id})=K_0 , &\text{in $\left( \R^d \setminus \overline{\Omega} \right)$}\cap \{w>\frac12|x|^2\}, \\ 
 w=h_0+\frac12|x|^2 ,&\text{on $ \partial \Omega    $}, \\
 |\nabla w-x| =\lambda_0, &\text{on $w=\frac12|x|^2$}. 
 \end{cases}
\end{equation}

\begin{figure}[htb]
\centering \def\svgwidth{300pt}
\input{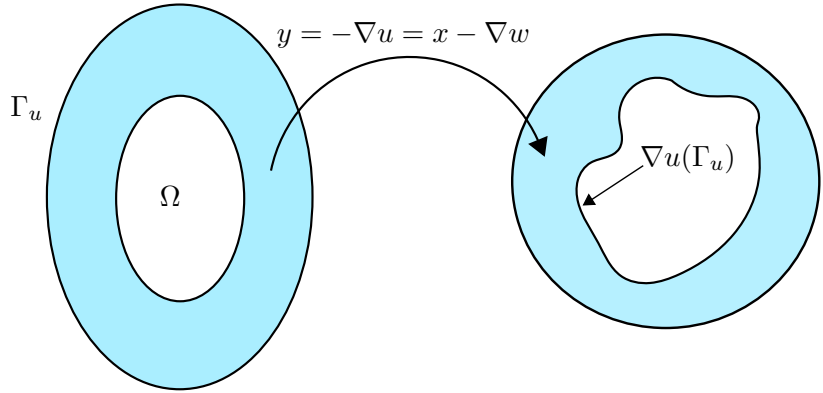}
\caption{\footnotesize The annular domains having one free and one prescribed component represent the reference and target sets under the transport map $y=-\na u$.   }
\label{Fig-ring}
\end{figure}

 The total mass of transport $m$ is also unknown and must be determined from $w$ by the formula
 \[m:=\hbox{Vol}\left(\po u\setminus \Om\right)=\frac1{K_0}\hbox{Vol}\left(\na(\po u\setminus \Om)\right).\]
To the best of our knowledge the study of these type of problems does not appear in the existing literature.

\section{The homogeneous problem}\label{sec-homogen}
In this section we prove Theorem A which concerns \eqref{prob-main} for $K_0=0$,
and this assumption will be in force throughout the section in any reference to \eqref{prob-main},
unless explicitly stated otherwise. The existence of a weak solution is proved in
Proposition \ref{prop-exist}, and the qualitative properties of weak solution are proved in 
subsection \ref{subsec-reg-uniq}. Moreover, the formula \eqref{h-star-def} gives an explicit
characterisation of this unique solution.



We start with a simple example, which is meant to illustrate how the solution to (\ref{prob-main})
looks like in the simplest case.

\begin{example}\label{ex-cones}{\normalfont(Truncated cones, cf. Figure \ref{Fig-cone})}
Let $\Omega$ be the unit ball of $\R^d$ $(d\geq 2)$ and fix $\lambda>0$.
Consider the problem \eqref{prob-main} with $K_0=0$, gradient condition $\lambda>0$, and boundary data identically 1.
Then, we can easily see (by direct computation)
that the function $u(x) = 1+\lambda - \lambda |x|$, with $1\leq |x| \leq 1+1/ \lambda$
defines a solution to \eqref{prob-main} which is concave on its positivity set (with a convention that $u$
is extended as identically 1 in $\Omega$). Clearly the free boundary will be the circle $|x|  = 1+1/\lambda $ 
and the free boundary condition in \eqref{prob-main} will be satisfied at all points.
Note that the graph of $u$ is a boundary of a solid truncated cone.

What we can also observe from this example, is that at any point of the free boundary,
there is a unique support plane to the graph of $u$ having slope $\lambda$. More precisely,
for any $M = (x_1, x_2,...,x_d)$ on the free boundary, the unit normal $\nu\in \mathbb{S}^d$ of this hyperplane has the form
$$
\nu = \frac{1}{(1+\lambda^2 )^{1/2}} \left(  - \frac{\lambda^2}{1+\lambda} x_1 , ..., - \frac{\lambda^2}{1+\lambda} x_d ,  -1  \right) \in \mathbb{S}^{d}.
$$

It is also clear, that for any ball $B(x_0, r)$, where $x_0\in \R^d$ and $r>0$, we can translate and scale $u$
constructed for $B(0,1)$ to get
a conical solution for a general ball too.
Namely, setting
$$
v(x) = 1+\lambda - \frac{\lambda}{r} | x-x_0|, \ \ \text{ where } r\leq |x-x_0| \leq r \left( 1+ \frac{1}{\lambda} \right),
$$
we get a solution to \eqref{prob-main} for $\Omega=B(x_0,r)$, $K_0 =0$, and the given $\lambda>0$.
The free boundary in this case will be the sphere $|x-x_0| = r\left(1+\frac{1}{\lambda} \right)$.

Finally, observe that a similar construction provides a convex solution to \eqref{prob-main}
with the same data, however with the free boundary contained inside $\Omega$.
\end{example}

\subsection{The case of polygonal domains}
The idea of the proof of Theorem A, is to first understand the case when $\Omega$ is a polyhedron,
and then use approximation by polyhedral domains to settle the general case.
In this section we study (\ref{prob-main}) when the domain $\Omega $ is a convex polyhedron,
i.e. a bounded convex domain which is an intersection of some finite number of halfspaces.
For the sake of construction and notation we identify $\Omega \subset \R^d$ with
$\Omega \times \{0\}\subset \R^{d+1}$.

\begin{prop}\label{Prop-polyhedron}
In \eqref{prob-main} let $\Om $ be a convex polyhedron. Then there is a generalized 
solution to \eqref{prob-main}.
\end{prop}

\begin{proof}
Set $\Sigma = \{ X\in \R^{d+1}: \ x_{d+1}= h_0 \}$, and
by $F_1,...,F_n$ denote the facets of $\Omega$, i.e. the ($d-1$)-dimensional flat portions of $\partial \Omega$.
For the orthogonal projection operator in the $(d+1)$-th direction
$$
\pi_{d+1}(X)=(x_1,...,x_d,0), \text{ where } X=(x_1,...,x_{d+1}) \in \R^{d+1}
$$
denote $\widehat{\Omega} = \pi_{d+1}^{-1}(\Omega) \cap \Sigma$, i.e. the lift of $\Omega$ into $\Sigma$.
Similarly, for each $1\leq i \leq n$ define $G_i = \pi_{d+1}^{-1}(F_i) \cap \Sigma$.
It is clear that $\widehat{\Omega}$ is a convex polyhedron in $\Sigma$ with facets $G_1,...,G_n$.

For $1\leq i \leq n$ consider a hyperplane $H_i=\{X\in \R^{d+1}: \ X \cdot \nu^{(i)} = c_i \}$ in $\R^{d+1}$, where 
$\nu^{(i)}$ is the unit normal and $c_i\in \R$. We choose $H_i$ so that it passes through the facet $G_i$,
has unit normal $\nu^{(i)}$ satisfying 
\begin{equation}\label{nu-d+1}
\nu^{(i)}_{d+1} = (1+\lambda_0^2)^{-1/2}
\end{equation}
and for its graph 
$$
u_i(x)= \frac{1}{\nu^{(i)}_{d+1} } \big[ c_i  - (\nu^{(i)}_1 x_1 +...+\nu^{(i)}_d x_d)  \big], \qquad x\in \R^d
$$
we have
$$
\Omega \subset \{(x,0) \in \R^d\times \{0\}: \ u_i(x) \geq h_0 \}.
$$
It is clear that such $H_i$ exists. Since $|\nu^{(i)}|=1$
from (\ref{nu-d+1}) we get
\begin{equation}\label{grad-u-i}
|\nabla u_i(x)  | =  \frac{\left(1- (\nu_{d+1}^{(i)} )^2 \right)^{1/2}}{\nu_{d+1}^{(i)}} =  \lambda_0 \text{ on } \R^d.
\end{equation}

We now define
$$
\mathcal{U}  = \bigcap\limits_{i=1}^n  \{ x\in \R^d: \ u_i(x) \geq 0   \} \setminus \Omega.
$$
Observe that $\mathcal U$ defines a bounded set in $\R^d$. For $x\in \mathcal{U}$ consider the lower envelope
\begin{equation}\label{u-min-def}
u(x) = \min\limits_{1\leq i \leq n} u_i(x).
\end{equation}
We thus have that $u$ is a piecewise linear function.
Moreover, it follows by (\ref{grad-u-i}) that 
$|\nabla u(x) | = \lambda_0 $ on $\partial \{x\in \R^d: u(x)>0  \} \setminus \overline{\Omega}$ away from a set of zero $\mathcal{H}^{d-1}  $-measure
(namely, away from the non-smooth boundary of the polyhedron $\{x\in \R^d:  \ u(x) =0 \} )$.
We also have $u= h_0 $ on $\partial \Omega$ by construction. 
Thus it is only left to prove that $u$ is a solution to the equation.
To this end define $\mathcal{K}$ to be the body bounded by the graph of $u$ and hyperplanes $x_{d+1}=0$ and $\Sigma$.
More precisely,
$$
\mathcal{K}= \{ (x,t) \in \R^d\times \R: \ x\in \overline{\mathcal{U}},  \ 0\leq t \leq u(x)  \} \cup \{ (x,t) \in \R^d\times \R: x\in \overline{\Omega}, \ 0\leq t \leq h_0  \}.
$$
By construction $\mathcal{K}$ is a convex polyhedron in $\R^{d+1}$.
Since $u$ is non-smooth, to show $\det D^2(-u)=0$ in $\{u>0\} \setminus \overline{\Omega}$,
we need to prove that the corresponding Monge-Amp\`ere measure is vanishing (see subsection \ref{subsec-weakSol}),
or equivalently  that the gradient image of $u$ on its positivity set has measure 0.
To see this, observe that at points where $u$ is smooth, its gradient mapping assumes only finitely many values,
which are precisely the gradients of $u_i$-s considered above.
We are thus left to treat the points where $u$ is not differentiable.
It is clear that $u$ is non-smooth at some $x \in \mathcal{U}$ iff $X=(x,u(x))$ lies on the
non-smooth boundary of the polyhedron $\mathcal{K}$. But the non-smooth boundary of
the polyhedron is the union of its $k$-dimensional edges, where $k=0,...,d-1$.
Along an edge $\mathcal{E}$ of dimension $k\geq 1$, the gradient mapping of $u$ is the same,
namely for any $x, y$ in the interior of  $\mathcal{E}$, we have $\omega_x(u) = \omega_y(u)$.
Hence, any $\alpha\in \omega_\mathcal{E}(u)$ lies in the gradient image of at least two different points,
and therefore by the celebrated result of Aleksandrov (see subsection \ref{subsec-weakSol})
we get that $|\omega_\mathcal{E}(u)|=0$. As a result we see that in order to complete the proof,
we need to show that $\mathcal{K}$ has no edges of dimension 0 (i.e. vertices)
in the strip $0<x_{d+1}<h_0$. In what follows we prove this statement and hence the proposition.
For the sake of clarity we split the proof into few steps.

\vspace{0.1cm}
\textbf{Step 1.} \emph{Vertex of $\mathcal{K}$ and a ray through it.}
\vspace{0.1cm}

Assume for contradiction, that $X^{(0)}=(x^{(0)},u(x^{(0)}))\in \R^{d+1}$ is a point on the graph of $u$
satisfying $0<X^{(0)} \cdot e_{d+1} < h_0$ and such that there is no line on the graph of $u$ passing through $X^{(0)}$. Then
there must be at least $d+1$ hyperplanes $H_i$ having a unique intersection point at $X^{(0)}$.
Consequently, the system of linear equations
\begin{equation}\label{unique-X0}
(X-X^{(0)})\cdot \nu^{(i)} =0 \ \ \text{ for all } \ i=1,...,d+1,
\end{equation} has a unique solution,
hence the collection $\{\nu^{(1)},...,\nu^{(d+1)}\}$ is linearly independent in $\R^{d+1}$.
This implies, that any $d$ out of these $d+1$ hyperplanes intersect in a line (e.g. by rank-nullity theorem).
We claim that at least one of these $d+1$ lines must intersect $\Sigma$.
Indeed, if not the case, then all $d+1 $ lines must be parallel to $\Sigma$,
and hence if $l_i$ is the direction vector of the $i$-th line,
then $l_i \cdot e_{d+1} =0$ for all $1\leq i \leq d+1$, where $e_{d+1}$ is the normal vector of $\Sigma$.
We have that $\{\nu^{(i)}\}_{i=1}^{d+1}$ is linearly independent,
and $e_{d+1}$ is a non-zero vector. Hence, we may replace one of $\nu^{(i)}$-s by $e_{d+1}$
and still get a collection of $(d+1)$-linearly independent vectors.
With this in mind, assume that $\{ \nu^{(1)},...,\nu^{(d)}, e_{d+1} \}$
is linearly independent, and let the line $L_1$ be the intersection of $H_1$,...,$H_{d}$.
By assumption, for $l_1$, the direction vector of $L_1$, we have $l_1 \cdot \nu^{(i)} =0$
for all $1\leq i \leq d$ since $L_1 \subset H_i$.
As $L_1$ does not intersect $\Sigma$, it follows that $l_1 \cdot e_{d+1} =0$.
We get that $l_1$ is orthogonal to $d+1$ linearly independent vectors, which is a contradiction
since $l_1$ is non-zero.
Hence, we conclude that at least one of the lines through $X^{(0)}$
formed as an intersection of $d$ planes must intersect $\Sigma$.
Let $L_0$ be the ray on this line intersecting $\Sigma$.

\vspace{0.1cm}
\textbf{Step 2.} \emph{Projection of $L_0$.}
\vspace{0.1cm}

Recall that each $H_i$ passes through the facet $G_i$ by construction.
At this stage we identify $G_i$ with its bounding $(d-1)$-dimensional hyperplane in $\Sigma$;
clearly $G_i = H_i \cap \Sigma$. Let $\widehat{L}_0$ be the orthogonal projection of $L_0$ onto $\Sigma$.
We claim that each point on $\widehat{L}_0$ is equidistant from all $G_1,..., G_d$.
To see this, we first observe that the equation of $H_i$ gives
\begin{equation}\label{G-i-plane}
G_i = \{ (x, h_0)\in \R^d \times \R : \ x \cdot \overline{\nu^{(i)}} - 
\overline{X^{(0)}} \cdot \overline{\nu^{(i)}}	 + (h_0-X^{(0)}_{d+1} )  \nu^{(i)}_{d+1} =0   \},
\end{equation}
where for $X=(x_1,...,x_{d+1})\in \R^{d+1} $ we have set $\overline{X} = (x_1,...,x_d) $. 
Now if $Z = (z,z_{d+1}) \in L_0$ is any, then for its projection $\widehat{Z} = (z, h_0) \in \widehat{L}_0$
and for each $1\leq i \leq d$ we have
\begin{equation}\label{dist-from-G-i}
 \mathrm{dist}( \widehat{Z}, G_i) = \frac{ | z \cdot \overline{\nu^{(i)}} - \overline{X^{(0)}} \cdot \overline{\nu}^{(i)} + (h_0-X_{d+1}^{(0)}) \nu_{d+1}^{(0)} |  }{ |\overline{\nu^{(i)}} | } = 
 \frac{|h_0 - z_{d+1}| |\nu_{d+1}^{(i)}| }{|\overline{\nu^{(i)}} |},
\end{equation}
where the second equality is due to the fact that $Z \in L_0 \subset H_i$ for any $1\leq i \leq d$, and condition (\ref{nu-d+1}). 
Applying (\ref{nu-d+1}) to (\ref{dist-from-G-i}) we get that $\mathrm{dist}( \widehat{Z}, G_i)$ is independent of $1\leq i \leq d$
for each given $Z\in L_0$.

\vspace{0.1cm}
\textbf{Step 3.} \emph{Round cones touching the facets.}
\vspace{0.1cm}

Let $Y_0 $ be the intersection of $L_0$ and $\Sigma$. By independence of $\nu^{(1)},...,\nu^{(d)}$,
the point $Y_0$ forms a vertex of the polyhedron $\widehat{\Omega}$. In particular we get that $\widehat{L_0} \cap int(\widehat{\Omega}) \neq \emptyset$.
Take $\widehat{Z} \in  \widehat{L_0} \cap int(\widehat{\Omega}) $ close to $Y_0$. Then by Step 2
we have that the ball $B_Z: = B( \widehat{Z}, \mathrm{dist}( \widehat{Z}, G_1) )  $ touches facets $G_1,...,G_d$ tangentially. 
We next consider a round cone $\mathcal{C}_Z$ which has the ball $B_Z$
as its base, and has vertex at $Z\in L_0$ which has $\widehat{Z}$ as its projection.
In view of the construction we have that the hyperplanes $H_1,...,H_d$ are support planes for $\mathcal{C}_Z$
all intersecting the boundary of the cone in line segments.
Moving the point $Z$ along the ray $L_0$,  we will get eventually, for some $Z $, that the ball $B_Z$ touches $G_{d+1}$ as well, 
thanks to the convexity of the region bounded by facets $G_1,...,G_{d+1}$.
Hence, the cone $\mathcal{C}_Z$ at that position will have $H_{d+1}$ as its support plane too,
as being a round cone of slope $\lambda_0$, it equals the lower envelope of its all support hyperplanes of slope $\lambda_0$.
But then, the ray $L_0$ has to intersect $H_{d+1}$ above $\Sigma$, producing a new intersection point of $H_1,...,H_{d+1}$.
The latter contradicts to the uniqueness of $X^{(0)}$, and hence completes the proof 
of the proposition.
\end{proof}

\begin{remark}
It is worthwhile to observe, 
that in the case when the all $d+1$ rays intersect $\Sigma$ (Step 1 in the proof), a different argument handles the proof above.
Indeed, in such a case we get that the facets $G_1$,...,$G_{d+1}$ have the property that any $d$ of them intersect in
a single point. Moreover all these $d + 1$ points are different, as if any two were the same,
that point would lie in the intersection of the hyperplanes $H_i$ but which we know contains
a single point, which is $X^{(0)}$. These properties imply that $\widehat{\Omega}$ is a $d$-simplex.
But a $d$-dimensional simplex has exactly $d + 1$ number of different facets of
dimension $d - 1$; we thus get that the collection $\{G_1, ...,G_{d+1}\}$ exhausts the all facets of $\widehat{\Omega}$.
Now the contradiction follows easily since there must be at least one $H_i$ whose graph has
value larger than $h_0$ at $X^{(0)}$. But by assumption all $H_i$-s are strictly less than $h_0$ at $X^{(0)}$ and
we get a contradiction, by so completing the argument.
\end{remark}

\subsection{Approximation argument}

Here we study the case of general convex domain. Let $\Omega$ be a bounded convex domain with $C^1$ boundary.
Thanks to the smoothness of $\partial \Omega$
for each $X_0\in \partial \widehat{\Omega} $ there exists a unique
support hyperplane to $\widehat{\Omega}$ at $X_0$ having slope $\lambda_0$;
call this plane $H_{X_0}$. Following the discussion in subsection \ref{subsec-FB}, each $H_{X_0}$
can be identified with a linear function over $\R^d$. 
With this in mind, consider the lower envelope of these support hyperplanes
\begin{equation}\label{h-star-def}
h_*(x) : = \inf\limits_{X_0 \in \partial \widehat{\Omega} } H_{X_0}(x), \qquad x\in \R^d.
\end{equation}
The infimum here does not collapse, thanks to the condition on the slopes of support planes,
and hence \eqref{h-star-def} defines a function locally bounded below on $\R^d$.
Indeed, let us see that the zero set of any $H_{X_0}$ stays on a uniform distance from $\Omega$.
Similar to the computations in the proof of Proposition \ref{Prop-polyhedron}, we can identify each $H_{X_0}$ having unit normal $\nu\in \R^{d+1}$
with a linear function $H(x) = - \frac{1}{\nu_{d+1}} [ x\cdot \overline{\nu} + X_0 \cdot \nu ]  $, where $x\in \R^d$,
and $\nu=(\overline{\nu}, \nu_{d+1})$ with $\nu_{d+1} = - (1+\lambda_0^2)^{1/2}$.
Then, if $x\in \overline{\Omega}$ we have $H(x) \geq h_0$ and for any $y\in \R^d$ where $H(y) =0 $ we get
$$
h_0 (1+\lambda_0^2)^{1/2} \leq (x- y) \cdot \overline{\nu} \leq |x-y| \hspace{0.02cm} |\overline{\nu}|,
$$
and hence, $h_*$ gives a locally bounded function which is concave as a lower envelope of concave functions.
We have the following result.

\begin{prop}\label{prop-exist}{\normalfont(Existence of a solution)}
Let $\Omega$ be a bounded convex domain in $\R^d$ with $C^1$ boundary.
Then $h_*$  is a weak solution to \eqref{prob-main}.
\end{prop}

\begin{proof}
We will approximate $h_*$ by polyhedral solutions, and will use the weak continuity of Monge-Amp\`ere measure.
To this end, fix  a sequence  $\{x^{(n)} \}_{n=1}^\infty \subset \partial \Omega $
forming a set of dense distinct points in $\partial \Omega$.
For each integer $n \geq n_0$, where $n_0\in \mathbb{N}$ is large enough, we let $\Omega_n$ be the convex polyhedron
bounded by tangent hyperplanes of $\Omega$ at points $\{x^{(1)},..., x^{(n)} \}$;
obviously $\Omega \subset \Omega_n$. Define $g_n:\partial \Omega_n \to \R$ by $g_n(x) = \pi_{d+1}^{-1}(x)\cap \Sigma$
where $x\in \partial \Omega_n$.
We now let $u_n$ be the solution of (\ref{prob-main}) constructed in Proposition \ref{Prop-polyhedron} for $\Omega_n$
and $g_n$. For notational convenience
we extend each $u_n$ into $\Omega_n$ having the same graph as the hyperplane $\Sigma$,
in this case as identically equal to constant $h_0$; this extension is also denoted by $u_n$. 
Observe as well, that by (\ref{u-min-def}) each $u_n$ is defined everywhere on $\R^d$.
Finally, after relabelling we assume that $n_0=1$.

By (\ref{u-min-def}) we have
\begin{equation}\label{u-n-decrease}
 u_1(x)  \geq ... \geq u_n(x) \geq ... \geq h_*(x) \qquad x\in \R^d.
\end{equation}
Since $u_n$ is decreasing and bounded below, it converges to some function $u_0$ defined on $\R^d$,
which is clearly concave.
We claim that the convergence is uniform in the set $\mathcal{U}_1:=  \{u_1\geq -1\}$;
this will follow from a well-known fact that convex functions are locally Lipschitz,
with Lipschitz constant admitting a bound by $L^\infty$ norm of the function.

Since the sequence $\{u_n\}$ is uniformly bounded on $\mathcal{U}_1$,
by \cite[Lemma 3.1]{Bak}
the set of gradient images of $\{u_n\}$ will be uniformly bounded on the set $\{ u_1 \geq 0 \}$,
which contains the sets $\{u_n \geq 0\}$ by (\ref{u-n-decrease}).
This in its turn implies uniform Lipschitz bounds on the sequence $\{u_n\}$ on the set $\{u_1\geq 0\}$ by \cite[Proposition 2.4]{Ra-Ta}.
Hence, Arzel\`{a}-Ascoli implies existence of a subsequence $\{n_k\}$ along which $u_{n_k}$ converges to $u_0$ uniformly
on the set $\{u_0 \geq 0\}$. But then, the monotonicity of $\{u_n\}$ implies uniform convergence of the entire sequence on $\{u_0 \geq 0 \}$.

From weak convergence of the Monge-Amp\`ere measure it follows that 
$|\omega_E(-u_0)|=0$ for any Borel $E\subset \{u_0>0\} \setminus \overline{\Omega}$, hence $u_0$ is a generalized solution
to $\det D^2 u=0$ in $\{u>0\} \setminus \overline{\Omega}$. We also get that $u_0= h_0$ on $\partial \Omega$
as a result of the construction.

We next verify the free boundary condition for $u_0$.
To accomplish this, fix any $z_0\in \fb{u_0}$
where the $(d-1)$-dimensional convex set $\fb{u_0}$ has a tangent hyperplane.
We only need to check the condition for such points on the boundary.
There is a sequence $z_n \in \fb {u_n}$
converging to $z_0$. In view of concavity of $u_n$ we may assume without loss of generality that
each  $z_n$ is a regular point for $\fb{u_n}$,
and hence the graph of $u_n$ will have a unique support hyperplane with slope $\lambda_0 $; let $H_n$
be this hyperplane. Up to passing to a subsequence, we may assume that unit normals of $H_n$ converge to some unit vector $\nu_0 \in \mathbb{S}^d$.
Let $H_0$ be the hyperplane of $\R^{d+1}$ through $z_0$ and with unit normal $\nu_0$. By construction we get that $H_0$ is the unique support hyperplane
to the graph of $u_0$ at $z_0$. 
Since all hyperplanes $H_n$ are tangent to the graph of $u_n$ and have slope $\lambda_0$,
we obtain that $H_0$ is tangent to the graph of $u_0$ at $z_0$ and has slope $\lambda_0$,
and hence the normal derivative of $u_0$ at $z_0$ equals $\lambda_0$,
which is the free boundary condition at $z_0$.

Finally, to complete the proof, it is left to show that $u_0 = h_*$
on the set $\omega:= \{h_*>0\}\setminus \Omega$.
To see this, let $\{X_n\}_{n=1}^\infty \subset \partial \widehat{\Omega}$ be the dense set of points
fixed above, for which we have $u_0 = \inf_n H_{X_n}$.
Then, clearly we get $h_*(x) \leq u_0(x) $ for all $x\in \omega$.
Now if for some $y_0 \in \R^d\setminus \Omega$ we get $h_*(y_0) = 0 $,
then there is a sequence of support hyperplanes $H_{Y_n}$ such that $H_{Y_n}(y_0) \to c_0 \leq 0$.
Due to the density of $X_n$, up to passing to a subsequence, we may assume that $|X_n - Y_n| \to 0$
hence $|H_{X_n}(y_0) - H_{Y_n}(y_0)| \to 0$, which follows from the fact that at each point of $\partial \widehat{\Omega}$
the support hyperplane with slope $\lambda_0$ is unique by $C^1$-smoothness of $\partial \Omega$.
We thus get $u_0(y_0)=0$ and conclude that $h_* = u_0$ on $\omega$.
The proof of the proposition is complete.
\end{proof}

\begin{remark}
For $\Omega$ without $C^1$ smoothness assumption, the proof of the previous proposition gives the
existence of a solution to \eqref{prob-main}, however we cannot claim that the solution obtained from polyhedral approximation must coincide with $h_*$.
\end{remark}

Our next result will be used to establish uniqueness of solutions to \eqref{prob-main}.

\begin{lem}\label{Lem-segment}{\normalfont{(Segments in the graph)}}
Let $u$ be any weak solution to \eqref{prob-main} which is concave.
Then, for any point $X^{(0)}=(x^0, u(x^0) ) \in \R^{d+1}$ on the graph of $u_0$,
there exists a line segment through $X^{(0)}$ lying entirely on the graph of $u$ 
and having endpoints at the free boundary and on $\partial \widehat{\Omega}$.
\end{lem}

\begin{proof}
Let $\mathcal{M}$ be the graph of $u$. We will first prove the lemma when $X^{(0)}$
is in the interior of $\mathcal{M}$, i.e. $0<u(x^0)<h_0$. 
In this case, let $\Pi$ be any support hyperplane of $\mathcal{M}$ at $X^{(0)}$,
and consider the convex hull  $\mathcal{X} = \mathrm{Hull} (\Pi \cap \mathcal{M})$.
Since $\Pi$ is a support plane and $u$ is concave, we get that $\mathcal{X}$
is a compact and convex set lying on  the graph of $u$.
To prove the lemma it is enough to see that $\mathcal{X}$ intersects the $h_0$- and $0$-level surfaces of $u$,
since then the statement of this lemma would follow by concavity of $u$.
Now assume for contradiction that $\mathcal{X}$ stays on a positive distance from 0-level surface of $u$, 
or equivalently from $\R^d\times\{0\}$. 
We enclose the convex body bounded by $\mathcal{M}$ and hyperplanes $x_{d+1}=0$ and $x_{d+1}=h_0$
by a convex polytope  having one of its facets lying on the support plane $\Pi$.
Moreover, we choose this polytope so that its boundary intersects $\mathcal{M}$ by $\mathcal{X}$ only.
For $\e>0$ small, applying \cite[Theorem 1.1]{Ghomi04} due to M. Ghomi
gives a convex body $D_\e$ with $C^\infty$ boundary such that
$\partial D_\e \cap \mathcal{M} = \mathcal{X} $,  $\partial D_\e \setminus \mathcal{X}  $ has positive curvature,
and the Hausdorff distance of $D$ and the fixed polytope does not exceed $\e$.
By choosing $\e>0$ small enough and relying on the fact that $\Pi$ is a support plane 
we can ensure that $\partial D_\e$ stays in between $\Pi$ and $\mathcal{M}$.
Consequently, in the neighbourhood of $X^{(0)}$ the gradient image of the hypersurface $\partial D_\e$ will be contained
in the gradient image of $\mathcal{M}$. But this is a contradiction,
since the latter has measure zero due to the equation, while the former has positive measure,
in view of the strict convexity of $\partial D_\e$ outside $\mathcal{X}$.
This contradiction proves the lemma when $X^{(0)}$ is in the interior of the graph.

To complete the proof we are left to cover the case when $X^{(0)}$ is on the free boundary or on $\partial \widehat{\Omega}$.
Assume the former, and fix a sequence of points $X^{(n)}$, $n\in \mathbb{N}$, in the interior of $\mathcal{M}$
converging to $X^{(0)}$. As we have proved already, for each $n$ there is a support plane $\Pi_n$ 
through $X^{(n)}$ intersecting the 0-level surface at some $A_n$ and the 1-level surface at some $B_n$.
Extracting a convergent subsequence, we may assume that the normal vectors $\nu_n$ of $\Pi_n$ converge to some $\nu_0$, and also that  $A_n \to A_0$, $B_n \to B_0$
where $u(A_0)=0$ and $u(B_0)=h_0$.
By compactness we get that the $A_ 0 =X_0$ and the hyperplane through $X^{(0)}$ and having $\nu_0$ as its normal
is a support plane of $\mathcal{M}$. Again, due to compactness we get that $B_0\in \Pi_0$,
hence by concavity of $u$ we have that the line segment $[X^{(0)}, B_0]$ lies on the graph of $u$.

The proof of the lemma is complete.
\end{proof}

We next establish a comparison principle with solutions having $C^1$ free boundaries or $C^1$ graphs.

\begin{lem}\label{Lem-comparison}\normalfont{(Comparison principle)}
Let $\Omega_1 \subset \Omega_2$ be two convex domains,
and let the concave function $u_i$ be a weak solution to (\ref{prob-main}) corresponding to $\Omega_i$, $i=1,2$.
Define $\omega_i$ to be the closure of the convex hull of $\Gamma_i: = \fb {u_i} $, for $i=1,2$.
We have the following:

\begin{itemize}
 \item[{\normalfont{(a)}}] if either of $\Gamma_i$ is $C^1$, then $\omega_1 \subset \omega_2$,
 \smallskip
 \item[{\normalfont{(b)}}] if either of $u_i$ is $C^1$ regular in $ \{ u_i>0\} \setminus \overline{\Omega}$, then one has
 $U_1\leq U_2$ on $\R^d$ where $U_i$ stands for the extension of $u_i$ into $\Omega_i$ as identically $h_0$.
\end{itemize}

\end{lem}

\begin{proof}
We start with part (a). It is enough to prove the claim
assuming that $\Gamma_1$ is $C^1$, since the other case follows a similar analysis.
Assume for contradiction that $\omega_1$ is not inside $\omega_2$.
Then, there is a point $X_2\in \Gamma_2$ such that $u_1(X_2)>0$,
i.e. $X_2$ is in the interior of $\omega_1$.
Since $ \Gamma_2$ is convex its normal exists almost everywhere,
and without loss of generality we will assume that at $X_2$ the set 
$ \Gamma_2$ has a unit outward normal, call it $\nu_2$.
Thanks to its smoothness at $X_2$ the set $\Gamma_2$  has a unique support
hyperplane at $X_2$, call it $G_2\subset \R^d\times \{0\}$
and hence, the graph of $u_2$ has a unique support plane with slope $\lambda_0$ at $X_2$;
which we denote by $H_2\subset \R^{d+1}$. 

Since $X_2$ is in the interior of $\omega_1$ translating $G_2$ and $H_2$
in the direction of $\nu_2$ we will get a point $X_1\in \Gamma_1$ 
where the shifted copies of $G_2$ and $H_2$, denoted respectively by $G_1$ and $H_1$,
form a support hyperplanes to $\Gamma_1$ and the graph of $u_1$.
Due to $C^1$ smoothness of $\Gamma_1$ we get that $G_1$ and $H_1$ are the only hyperplanes with those properties at $X_1$.
Now recall that $H_2$ supports the graph of $u_2$, hence it has $\widehat{\Omega}_1$ on one side of it.
But then, since $H_1$ is parallel to $H_2$ and is on the halfspace determined by $H_1 $ and not containing $\Omega_1$,
we get that $H_1$ cannot intersect $\widehat{\Omega}_1$. The latter contradicts the statement of Lemma \ref{Lem-segment}
as there should be a segment in the intersection of $H_1$ with the graph
of $u_1$ joining the free boundary with the boundary of $\widehat{\Omega}_1$.
The contradiction completes the proof of part (a).

The proof of part (b) follows from a similar idea, where we use support hyperplanes of the graph, and rely on the $C^1$-smoothness of one of the solutions.
Thus, the proof of the lemma is complete.
\end{proof}

\begin{figure}[htb]
\centering \def\svgwidth{300pt}
\input{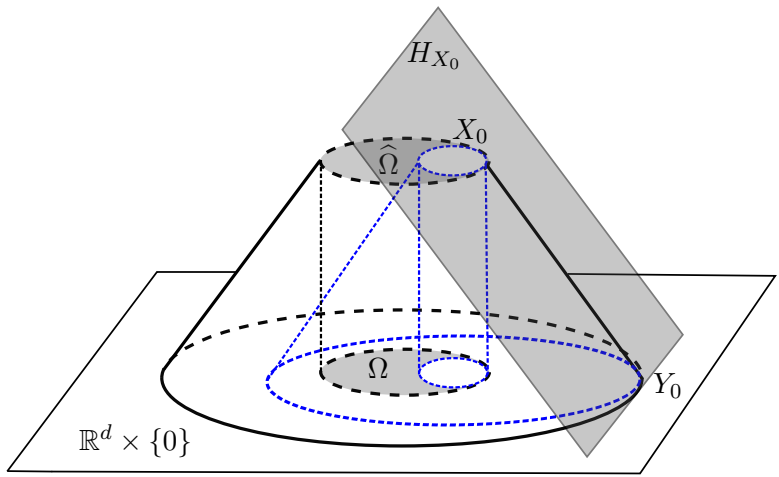}
\caption{\footnotesize{A schematic view of the construction of the conical barrier
from below. We take a ball $B\subset \Omega$ touching $\partial \Omega$
and following Example \ref{ex-cones} construct the conical solution $u_{X_0}$ having slope $\lambda_0$,
depicted in blue. The construction enforces $u_{X_0}$ and the solution $u$
to share the same supporting hyperplane at $X_0$ having slope $\lambda_0$.
This hyperplane, which is coloured in gray, intersects a line segment from the 
graphs of both $u$ and $u_{X_0}$. At the endpoint of this shared segment, denoted by $Y_0$,
the free boundary of the cone $u_{X_0}$ touches the free boundary of $u$.}}
\label{Fig-cone}
\end{figure}

\subsection{Regularity and geometry of the solution and the free boundary}\label{subsec-reg-uniq}

In this section we study uniqueness and qualitative properties of weak solutions
 to \eqref{prob-main} for $K_0=0$.

\begin{lem}\label{Lem-C11}
Let $\Omega$ be a bounded convex domain with $C^{1,1}$-smooth boundary,
and let $h_*$ be defined as in \eqref{h-star-def}. Then,
$\fb{h_*} $ is $C^{1,1}$-regular, and $h^* \in C^{1,1} (\R^d\setminus \overline{\Omega})$.
\end{lem}

\begin{proof}
We start with the regularity of the free boundary. 
To this end, relying on $C^{1,1}$ regularity of $\Omega$ we fix $r_0>0$ such that
for any $X\in \partial \Omega$ there is a ball $B= B(Z, r_0)$ where $Z\in \Omega$
with the property that $B$ touches $\partial \Omega$ at $X_0$ and stays completely inside $\Omega$.
Now let $u_{X_0}$ be the conical solution corresponding to $B$
constructed in Example \ref{ex-cones}, with the same parameters $h_0$ and $\lambda_0$.
The free boundary of $u_{X_0}$ is a ball by construction, and $u \in C^{1,1} (\R^d \setminus \overline{B}) $.
For $H_{X_0}$ - the support hyperplane of $\widehat{ \Omega}$ at $X_0$ having slope $\lambda_0$,
we know that it intersects the graph of $u_{X_0}$ by a line segment, call it $I=[X_0, Y_0]$
with $Y_0\in \partial \{u_{X_0} >0 \}\setminus \overline{B}$ (see Figure \ref{Fig-cone}).
Relying on the smoothness of $u_{X_0}$ and applying the comparison principle of Lemma \ref{Lem-comparison}
we see
that
$$
H_{X_0} = u_{X_0} \leq h_* \leq H_{X_0} \qquad \text{on the interval } I,
$$
hence the segment $I$ lies on the graph of $h_*$. Moreover, at the point $Y_0$ of the free boundary of $h_*$
we get a touching ball by construction. Moving $X_0$ on $\partial \widehat{\Omega}$
we conclude that the graph of $h_*$ consists of line segments cut out from the cones,
in particular we will cover all free boundary points of $h_*$.
This shows that the free boundary $\fb {h_*} $ is $C^{1,1}$ regular.

\smallskip

A similar argument shows that $h_*$ is $C^{1,1}$ in the interior.
Indeed, the argument above shows that all level surfaces of $h_*$ are $C^{1,1}$ uniformly.
Fix any $x_0$ where $h_*(x_0)>0$, and consider the line segment $I$ lying on the graph of $h_*$ and passing through $X_0:=(x_0, h_* (x_0))$.
At each point of the segment $I$ there is a ball $B$ with radius independent of the point,
such that $B$ lies in $ \{  x \in \R^d: \ h_*(x) =  h_*(x_0)  \}  $. Since the radii of these balls admit a uniform lower bound,
sliding the balls across $I$, we get a cylinder lying completely inside $\{  h_* >0\}$.
In particular there is a $(d+1)$-dimensional ball touching the graph of $h_*$ at $X_0$ and staying completely under the graph of $h_*$.
This gives that $h_* \in  C^{1,1} ( \R^d \setminus \overline{\Omega}) $,
and completes the proof of this lemma.
\end{proof}

With the regularity of $h_*$ at hand, we are in a position to address the uniqueness of solutions to \eqref{prob-main}.

\begin{lem}\label{Lem-uniq-0case}{\normalfont{(Uniqueness)}}
Assume $\Omega$ is a bounded convex domain with $C^{1,1}$ boundary.
Then \eqref{prob-main} has a unique weak solution given by \eqref{h-star-def}.
\end{lem}
\begin{proof}
Proposition \ref{prop-exist} already gives that $h_*$ is a solution, and we only need to establish uniqueness.
Let $u$ be any weak solution to \eqref{prob-main} which is concave. 
By Lemma \ref{Lem-C11} we have that $h_*$ is $C^{1,1}$ in the interior of its positivity set
and has $C^{1,1}$ regular free boundary.
Hence, the comparison principle of Lemma \ref{Lem-comparison} applied (twice) implies $u=h_*$ in $\{h_*>0\}$,
and finishes the proof of the lemma.
\end{proof}

\begin{remark}
The last lemma shows that for regular $\Omega$ the unique solution of \eqref{prob-main} is a ruled surface.
\end{remark}

We conclude this section, by 
showing that the strict convexity of $\Omega $ is inherited by the free boundary.
In order to quantify the strict convexity, we recall the definition of a \emph{rolling ball}
\`{a} la Blaschke.
\begin{defn}\label{def-rolling-ball}
We say that $\Om $ \emph{rolls freely} inside a ball of radius $R$,
if at each $X\in \partial \Om$ there exists $Z\in \R^d$ such that the ball $B=B(Z,R)$
contains $\Om$ and touches $\partial \Om$ at $X$. 
\end{defn}

For instance, if $\Om$ is $C^2$ and all principal curvatures of $\partial \Om$
are bounded below by some constant $\kappa>0$, then in view of the Blaschke's rolling ball theorem 
(which gives an inclusion principle, based on comparison of second fundamental forms)
we get that the constant $1/\kappa$ serves as a radius of the rolling ball in question.

\begin{lem}\label{Lem-reg-FB-0case}
Assume $\Omega$ is a bounded strictly convex domain with $C^{1,1}$ boundary.
Then the free boundary $ \fb{h_*} $ is strictly convex.
Moreover, if $\Om$ rolls freely inside a ball of radius $R$, then the free boundary has a rolling ball of 
radius
\begin{equation}\label{r1}
\frac{R}{\lambda_0} \left( \lambda_0 + h_0 \right),
\end{equation}
where the parameters $\lambda_0$ and $h_0$ are fixed from \eqref{prob-main}.
\end{lem}

\begin{proof}
In view of Lemma \ref{Lem-C11} we have that the free boundary of $h_*$
is $C^{1,1}$. Now assume for contradiction, that there is a point $X\in \fb{h_*} $ 
such that the unique support hyperplane of $\fb{h_*} $ at $X$,
call it $G$, intersects $\fb{h_*} $ in a point $Y\neq X$. By convexity, the line segment $[X, Y]$
should lie on the free boundary $\fb{h_*} $. 
Let $\Pi$ be the unique support hyperplane to the graph of $h_*$
passing through $X$. Clearly, $Y\in \Pi$, and by Lemma \ref{Lem-segment}
there exist points $\widehat{X}, \widehat{Y} \in \partial \widehat\Om$ such that
the segments $[X, \widehat{X}] $ and $[Y, \widehat{Y}]$ lie on the intersection of $\Pi$ and
the graph of $h_*$. In view of the strict convexity of $\Om$, we get $\widehat{X} = \widehat{Y}$.
But recall that the interior touching cone of Lemma \ref{Lem-C11} must have both
segments through $X$ and $Y$ on its graph, which is a contradiction,
and hence the strict convexity of $\Gamma$.

\smallskip

We now proceed to the second part of the lemma. Let $X\in \partial \Omega$ by any,
and $B=B(Z, R)$ be a ball containing $\Omega$ and touching $\partial \Omega$ at $X$.
Fix also an inner touching ball at $X$, call it $b$, as was discussed above in the proof of Lemma \ref{Lem-C11}.
Then, consider the conical solutions corresponding to $b$ and $B$, denoted by $u_b$ and $u_B$ respectively.
By construction, $h_*$, $u_b$ and $u_B$ all share a line segment on their graphs, determined by the intersection
of the support plane to $\widehat{\Omega}$ at $X$ and having slope $\lambda_0$.
It follows, by comparison given by Lemma \ref{Lem-comparison} that the free boundary
of $u_B$, which is a ball of radius \eqref{r1} as can be seen
from Example \ref{ex-cones}, contains $\fb{h_*}$ in its interior and touches it at the endpoint of the shared line segment. 
But as we saw in the proof of Lemma \ref{Lem-C11}, moving this construction along the boundary of $\Omega$
we cover all free boundary points of $h_*$, and hence the claim.
\end{proof}

\bigskip

\begin{proof}[Proof of Theorem A]
Proposition \ref{prop-exist} provides the existence of a solution and
Lemma \ref{Lem-uniq-0case} establishes the uniqueness.
Next, we get that this unique solution is a ruled surface in view of Lemma \ref{Lem-segment},
and the regularity of the solution and the free boundary follows from
Lemma \ref{Lem-C11}. The strict convexity of the free boundary is due to Lemma \ref{Lem-reg-FB-0case}.
The proof of the theorem is complete.
\end{proof}

\section{Strictly elliptic case}\label{sec-nonZero}

Here we study the problem \eqref{prob-main} with $K_0>0$, and prove Theorem B.
The positivity of $K_0$ covers in particular the cases of positive curvature measure,
and positive Gauss curvature. One crucial difference, however,
from the homogeneous case, is that here \eqref{prob-main} may have no solutions
for arbitrary values of parameters involved in \eqref{prob-main}, 
as we show in the Appendix to this paper.

Throughout the section $\Omega$ is assumed to be $C^2$.


\subsection{Construction of a super-solution and Perron's method}\label{subsec-existence}
In this subsection we construct a solution to \eqref{prob-main} via Perron's method.
The set of all super-solutions to \eqref{prob-main} in a sense of Definition \ref{defn-Aleks}
will be denoted by $\mathbb{W}_+(K_0, \lambda_0, \Om)$.
We will prove the existence of a solution to \eqref{prob-main}
by showing that the infimum over all super-solutions solves \eqref{prob-main}.
As we show in the Appendix,  there are arrangements of the parameters involved in the formulation
of the problem \eqref{prob-main}, for which no solution exists.
We thus start by formulating a technical condition under which the existence result will be established.

Let $\kappa_0>0$ be the minimum over all $x\in \partial \Omega$ of the smallest principal
curvature of $\partial \Omega$ at $x$, and consider the inequality
\begin{equation}\label{ineq-star}
 K_0^{\frac 1d} \leq \psi^{-\frac 1d} (\lambda_0) \frac{  \kappa_0 \lambda_0^2 }{ h_0 \kappa_0 +  ( \lambda_0^2 + h_0^2 \kappa_0^2 )^{1/2} }.
\end{equation}

\begin{lem}\label{Lem-subsol-exists}
Assume the function $\psi=\psi(\xi)$ in \eqref{prob-main} is non-decreasing and  smooth.
Then, provided \eqref{ineq-star} holds, the set of super-solutions $\mathbb{W}_+$ is non-empty.
\end{lem}

\begin{proof}
We will construct a super-solution as a lower envelope of a certain
family of paraboloids. Set $r_0= 1/\kappa_0$, then
for a given $x_0\in \partial \Omega$, there is $z_0\in \R^d$ such that the ball $B= B(z_0, r_0)$
is internally tangent to $\Omega$ at $x_0$ (i.e. $B$ and $\Omega$ are tangent at $x_0$ and share the same outward normal at $x_0$).
Moreover, for any $x\in \partial \Omega$ and any $x'\in \partial B$
satisfying $n_\Omega(x) = n_B(x')$ for the interior normal vectors, we have
$II_{x, \Omega} \geq II_{x', B}$ for the second fundamental forms in view of the choice of $r_0$.
Applying the inclusion principle for ovaloids due to Rauch \cite{Rauch},
which is a generalization of Blaschke's rolling ball theorem, we conclude that $\Omega \subset B$.

Consider a paraboloid $P(x) = h_0+ \alpha r_0^2 - \alpha |x-z_0|^2$ where $\alpha>0$ is a constant to be determined below, and $r_0,z_0 $ are as above.
It is clear that $P(x_0)= h_0 $ and $P\geq h_0$ everywhere on $ \overline{\Omega} $
due to the inclusion $\Omega \subset  B$.
We now choose $\alpha$ so that $P$ will satisfy the first and the third conditions of 
Definition \ref{defn-Aleks}.
A straightforward computation shows that the free boundary condition for $P$ is equivalent to
\begin{equation}\label{lambda-0}
2 \alpha \left(  \frac{h_0 + \alpha r_0^2}{\alpha} \right)^{\frac 12} =\lambda_0 \quad \Rightarrow \quad (4r_0^2)\alpha^2+4h_0\alpha-\lambda_0^2=0.
\end{equation}
Solving the quadratic equation for positive $\alpha$, and replacing $r_0 = 1/\kappa_0$  we get
\begin{equation}\label{alpha-for-P}
 \alpha = \frac{\lambda_0^2  \kappa_0 }{ 2 \left( h_0 \kappa_0 + ( \lambda_0^2 + h_0^2 \kappa_0^2 )^{1/2}    \right)  }.
\end{equation}
Using the fact that $\psi$ is non-decreasing, we observe that $P$ satisfies the inequality in Definition \ref{defn-Aleks} (a) if 
$$
(2\alpha)^d \geq K_0 \psi \left( 2\alpha \left( \frac{h_0 + \alpha r_0^2}{\alpha} \right)^{1/2} \right) = 
K_0 \psi(\lambda_0),
$$
where we have also used \eqref{lambda-0}.
The obtained condition on $\alpha$ is precisely the inequality \eqref{ineq-star}.
The conclusion is that a single paraboloid satisfies a definition of a super-solution 
except for the boundary data. We next fix the boundary data, by taking the minimum over all
paraboloids. To this end fix a dense sequence of points $\{x_n\}_{n=1}^\infty \subset \partial \Omega$ 
and for each $n\in \mathbb N$ let $P_n$ be the paraboloid constructed above for $x_n$.
For $n=1,2,...$ define
$$
u_n(x) = \min_{1\leq k \leq n} P_k(x), \qquad x\in \R^d .
$$
Clearly, each $u_n$ is concave as a minimum of concave functions.
It is also clear, due to the construction of paraboloids, that $u_n(x_i) = h_0$ for each $1 \leq i \leq n$,
and that $u_n(x) \geq h_0$ for all $x\in \overline{\Omega}$.
We claim that $\det D^2 (-u_n) \geq K_0 \psi(|\nabla u_n|) $ on $\mathcal{U}_n:= \{u_n>0 \} \setminus \overline{\Omega} $.
Indeed, assume $n=2$, and fix any Borel set $E\subset \mathcal{U}_2$.
We partition $E$ into the following subsets: $E_1= \{x\in E: \ P_1(x)< P_2(x)\}$, $E_2= \{x\in E: \ P_2(x)< P_1(x)\}$, and $E_{12}= \{x\in E: \ P_1(x) = P_2(x)\}$.
It is clear, in view of continuity of $P_1$, $P_2$, that all sets $E_1, E_2, E_{12}$ are Borel.
For a given concave function $f$ and Borel set $B$, let $\omega_B(f)$  be the gradient image of $f$ on $E$.
Then, we get 
\begin{equation}\label{omega-for-u2}
\omega_E(-u_2) = \omega_{E_1}(-u_2) \cup \omega_{E_2}(-u_2)\cup \omega_{E_{12} }( -u_2). 
\end{equation}
Now observe, that thanks to the theorem of A.D. Aleksandrov (see subsection \ref{subsec-weakSol}), coupled with the fact that the sets $E_1$, $E_2$, $E_{12}$ are pairwise disjoint, we obtain
that the sets in the right-hand side of \eqref{omega-for-u2} intersect in null-sets.
Also note that $\omega_{E_1} (-P_i) \subset \omega_{E_1} (-u_2)$ and $\omega_{E_{12}} (-P_i) \subset  \omega_{E_{12}} (-u_2)$ for $i=1,2$.
Hence, we get
\begin{multline}\label{omega-est}
\int_{\omega_E(-u_2)  } \frac{d \xi}{\psi(|\xi|) }  =  
\int_{\omega_{E_1}(-u_2)} + \int_{\omega_{E_2}(-u_2)} +\int_{\omega_{E_{12}}(-u_2)} \geq \\
\int_{\omega_{E_1}(-P_1)} + \int_{\omega_{E_2}(-P_2)} +\int_{\omega_{E_{12}}(-P_2)} \geq \\
 K_0 |E_1| + K_0|E_2| + K_0 |E_{12}| = K_0 |E|.
\end{multline}
This shows that $u_2$ satisfies Definition \ref{defn-Aleks} on $ \mathcal{U}_2$;
the case for general $n\geq 2$ follows from the same argument by induction.
We conclude that
\begin{equation}\label{u-n-K0}
\det D^2 (-u_n) \geq K_0 \psi(|\nabla u_n|) \text{ on } \{u_n>0\} \setminus \overline{\Omega}, \qquad \text{ for each } n\in \mathbb N.  
\end{equation}
Due to the construction, for each $n\in \mathbb{N}$ we also have
\begin{equation}\label{u-n-FB}
 | \nabla u_n(x) | = \lambda_0, \qquad x\in \partial\{u_n>0\},
\end{equation}
in a weak sense, see subsection \ref{subsec-FB}.

Consider the set $\mathcal{U}_0: = \bigcap\limits_{n=1}^\infty \{u_n \geq -1\}$. In view of the construction
of $u_n$, we have that $\mathcal{U}_0$ is a compact convex set containing $\Omega$.
The sequence $\{u_n\}$ is a decreasing sequence of concave functions bounded below on $\mathcal{U}_0$,
and hence the limit $u_0(x): = \lim\limits_{n\to \infty} u_n(x)$, $x\in \mathcal{U}_0$ exists and defines a concave
function on $\mathcal{U}_0$. It is clear from the density of the points $\{x_n\}$ in $\partial \Omega$
that $u_0\equiv h_0$ on $\partial \Omega$, and also that $u_0\geq h_0$ on $\Omega$.
A similar argument as in Proposition \ref{prop-exist} shows that the convergence of $u_n$ to $u_0$
is uniform on the set $\overline{ \{u_0 >0 \} }$. In particular, 
from the weak continuity of the generalized Monge-Amp\`ere measure (see \cite[Theorem 9.1]{Bak})  and \eqref{u-n-K0} we get
that $\det D^2 (- u_0) \geq K_0 \psi(|\nabla u_0|)$ on $\{u_0>0\} \setminus \overline{\Om}$.
Next, the free boundary condition for $u_0$ as claimed in Definition \ref{defn-Aleks} (c)
follows from the same line of argument as we had in Proposition \ref{prop-exist} coupled with \eqref{u-n-FB}.
Finally, redefining $u_0$ as identically equal to the constant $h_0$ on $\Omega$,
provides an element of $\mathbb{W}_+$, and completes the proof of this lemma.
\end{proof}

\begin{remark}
Instead of paraboloids, one can construct a super-solution
as an envelope of spherical caps, via a similar procedure. While this seems to be a natural choice
for the equation of prescribed Gauss curvature, spherical caps
need more restrictive assumptions, as compared to \eqref{ineq-star},
for adjusting the free boundary condition.
\end{remark}

The next lemma is necessary for establishing a non-degeneracy result for super-solutions.

\begin{lem}\label{Lem-non-degeneracy}{\normalfont(Comparison with vanishing curvature case)}
Let $u$  be any weak super-solution to \eqref{prob-main}, and let $u_0$ be the unique
weak solution to \eqref{prob-main} with identically 0 r.h.s., and with the same $\Omega$, $h_0$ and $\lambda_0$. 
Then $u \geq u_0$ on $\{u_0 \geq 0\} \setminus \Omega$.
\end{lem}
\begin{proof}
Since $\Omega $ is assumed to be $C^2$ we have from Section \ref{sec-homogen} that $u_0$ is $C^{1,1}$
in its positivity set, and the free boundary $\partial\{u_0>0\} \setminus \overline{\Omega}$ is $C^{1,1}$ also.
Now, the proof of the lemma follows from precisely the same argument as we had for Lemma \ref{Lem-comparison}.
\end{proof}

In view of Lemma \ref{Lem-subsol-exists}, the set of super-solutions $\mathbb{W}_+$ is non-empty;
and thanks to Lemma \ref{Lem-non-degeneracy} all super-solutions are bounded below by the solution
of the homogeneous equation with the same input data; in particular, super-solutions do not degenerate (collapse).
Due to these observations, the lower envelope of $\mathbb{W}_+$
$$
u_*(x): = \inf\limits_{w\in \mathbb{W}_+} w(x), \qquad x\in \R^d,
$$
defines a concave function bounded below by $u_0$-the solution to the homogeneous equation,
and bounded above by the constant $h_0$. The aim is to show that $u_*$
is a solution to \eqref{prob-main}. This we will do in two steps, first showing that
$u_*$ solves the equation in \eqref{prob-main}, and second, which is the hardest part,
that $u_*$ satisfies the free boundary condition.

\begin{lem}\label{Lem-smallest-solves-PDE}
Let $u_*$ be as above. Then, 

\begin{itemize}
\item[$\bf 1^\circ$]
$u_* \in \mathbb{W}_+$ and
$$
\mathrm{det} D^2 (-u_*)= K_0 \psi( | \na u_* |) \ \ \ \text{ in } \ \ \ \{u_*>0\} \setminus \overline{\Omega}, 
$$
\item[$\bf 2^\circ$] the graph of $u_*$ is strictly concave and consequently $u_*\in C^\infty$ in $\po u_*\setminus \overline \Om$.
\end{itemize}
\end{lem}

\begin{proof}
${\bf 1^\circ}$ \ We start by showing the existence of a minimizing sequence from $\mathbb{W}_+$
converging to $u_*$. Consider the set $U:= \{u_*\geq -1\}$ and $U_0: = \{u_*\geq 0\}\Subset U$, and fix any $\e>0$.
All functions $w\in \mathbb{W}_+$ have uniform modulus of continuity in the set $U_0$
in view of \cite[Lemma 3.1]{Bak}. Hence, we can fix $\delta=\delta(\e)>0$ such that
for all $w\in \W_+ \cup \{u_*\}$ we get 
\begin{equation}\label{z1}
|w(x)-w(y)|\leq \e,  \text{ if } x,y \in U_0 \text{ and } |x-y|\leq \delta.
\end{equation}
Next, we fix a set of points $\{x_i: \ 1\leq i \leq n\} \subset U_0$ forming a $\delta$-net,
i.e. for any $x\in U_0$ there is $x_i$ with $|x-x_i|\leq \delta$.
Now for each $1\leq i \leq n $ we fix $w_i\in \mathbb{W}_+$ such that
$0\leq w_i(x_i) - u_*(x_i) \leq \e$. Since $w_\e : = \min_{1\leq i \leq n}w_i(x) \in \mathbb{W}_+$
we then get $0\leq w_\e(x_i) - u_*(x_i) \leq \e$ for all $1\leq i \leq n$.
This, together with \eqref{z1} easily implies that $0\leq w_\e(x) - u_*(x) \leq C \e$
for all $x\in U_0$ where $C>0$ is an absolute constant;
and hence the existence of a minimizing sequence.

Using the existence of a minimizing sequence and weak continuity of Monge-Amp\`ere measure,
as we did in the proof of Lemma \ref{Lem-subsol-exists}, we get $u_*\in \mathbb{W}_+$,
and to complete the proof of the current lemma, it is left to show that $u_*$
solves the PDE.

We have $\det D^2 (-u_*) \geq K_0 \psi(|\nabla u_*|) $ on $\{u_*>0\} \setminus \overline{\Omega}$, 
and assume, for contradiction, that there is a ball $B\subset \{u_*>0\}$
such that $\int_{\omega_B(-u_*)} \frac{d \xi}{\psi(|\xi|)} > K_0 |B|$. 
Let $u$ be the concave solution to $\det D^2 (-u)  = K_0 \psi(|\nabla u|)$ in $B$ and $u=u_*$ on $\partial B$.
The existence of the solution $u$ to Dirichlet's  problem in strictly convex domain for continuous boundary data 
is well known (see \cite[Theorem 10.1]{Bak}).
Now consider the function $\widetilde{u}$ equal to $u$ on $B$ and $u_*$ otherwise.
By construction $\widetilde{u}$ is concave, and following the same lines as in \eqref{omega-est}
we have that $\widetilde{u} \in \mathbb{W}_+$. Due to the comparison principle, we get $u< u_*$ in $B$, hence $\widetilde{u}$
is smaller than the infimum $u_*$ in $B$ which gives a contradiction
and completes the proof of part $\bf 1^\circ$.

\smallskip

$\bf 2^\circ$
Since $u_*$ solves the PDE by $\bf 1^\circ$,
to show its smoothness it is enough to prove that $u_*$ is strictly
concave (see \cite{Pog-MA}, and \cite{Figalli}).
To this end, we use an argument from \cite[section 5.6]{TW}. If $\graph{u_*}$ is not strictly concave then there is a line segment $\ell\subset \graph{u_*}$. 
Since $u_*$ solves the PDE in view of part $\bf 1^\circ$, there are two positive constants $\mu_1, \mu_2$ such that 
$\mu_1\le \det D^2(-u_*)\le \mu_2$. 
Then, it follows from \cite{Caff-loc} that $\ell$ cannot have extremal points in the relative 
interior of $\graph{u_*}$. Hence there is $z_1\in \p \widehat{\Om}$ 
such that $z_1\in \p \widehat{\Om} \cap \ell$.
Choosing a new coordinate system we can assume that $z_1=0$ and $\ell$ is the $x_d$ axis. 
Observe that $\ell$ is transversal to $\p \widehat{\Om}$. From the regularity of $ \Om$ it follows that for the 
convex function $\widetilde u$ representing the surface in the new coordinate system 
\begin{equation}\label{blya-ff}
0\le \widetilde u(x_1, \dots, x_{d-1}, x_d)\le C\sum_{i=1}^{d-1} x_i^2,
\end{equation}
near the origin for $x_d>0$ sufficiently small.
However, using the positivity of the right hand side of the PDE, one can construct 
a sub-solution $w$ near the origin, in a small ball $B:=B(\e e_d, \e)$, $\e>0$, say,
with $w\in C^2(\overline{B})$ and $\widetilde{u} \geq w \geq 0 $ in view of the comparison principle.
But the latter violates \eqref{blya-ff}, as 
$\p^2_{dd} w(0)>0$ thanks to the smoothness of $w$ and condition $\det D^2 w >0$ in $B$.
This completes the proof of part $\bf 2^\circ$, and hence of the lemma.
\end{proof}

\subsection{Blaschke extension of the minimal super-solution}
The aim of this subsection is to show that the smallest super-solution $u_*$
verifies the free boundary gradient condition. 
To this end, we will extend $u_*$ in a neighbourhood of each free boundary point
below the halfspace $\R^d\times \{0\}$ as a super-solution to our problem.
This extension will reduce the matters to a similar scenario as we had in the interior case.
The construction of the extension of $u_*$ is inspired by Blaschke's rolling ball theorem, and hence
the name of the extension. 

\smallskip

In the construction below all 
support hyperplanes of $u_*$ at the free boundary points are {\bf extreme},
that is the hyperplanes cannot be further tilted towards $\graph{u_*}$ (see subsection \ref{subsec-FB}).
Take any point $x \in \fb{u_*}$, if it is a regular point of the free boundary,
then there is a unique support hyperplane in $\R^{d+1}$ to the graph of $u_*$
passing through $X = (x, 0) \in \R^{d+1}$ and having slope $\lambda_0$.
We call this plane $H_x$. Otherwise, if $x$ is not a regular point, 
we do the construction that follows for any support hyperplane to $\graph{u_*}$
through $X$ having slope\footnote{Taking a sequence $x_d \in \fb{u_*}$ of regular points converging to $x$
and defining $H_x$ as the limit of the corresponding hyperplanes $H_{x_n}$, produces
a support hyperplane with slope $\lambda_0$ at a non-regular point $x$.} $\lambda_0$.
Next, let $H_{x}^\bot$ be the hyperplane passing through $H_{x} \cap (\R^d\times \{0\})$
and the normal to $H_{x}$. Clearly, $H_{x}^\bot$ is orthogonal to $H_{x}$.
We now define the convex body 
$\mathcal  S^+_*$ bounded by $\graph{u_*}$ if $0<x_{d+1}<h_0$, 
$\Om\times \{h_0\}$ if $x_{d+1}=h_0$, and when $x_{d+1}<0$
we define $\mathcal  S^+_*$ as the intersection of all half spaces corresponding to the 
extreme support functions at $\fb {u_*}$, which contain $\Omega$.

For the given $x\in \Gamma_{u_*}$ consider $\mathcal S_{x}^-$ as the mirror reflection of $\mathcal S^+_*$ with respect to $H_x^\bot$ (see Figure \ref{Fig-Blaschke}).
Let $x_1, \dots x_k$ be  distinct points on $\Gamma_{u_*}$ near fixed $x_0\in \Gamma_{u_*}$ and let 
$$
\mathcal S^m= \mathcal{S^+_*} \cap \bigcap\limits_{j=1}^m \mathcal S_{x_j}^-.
$$
Note that $x_j\in \mathcal S^m$ for every $j=1, \dots, m$ and $\mathcal S^m$ is a bounded convex body.

\begin{figure}[htb]
\centering \def\svgwidth{300pt}
\input{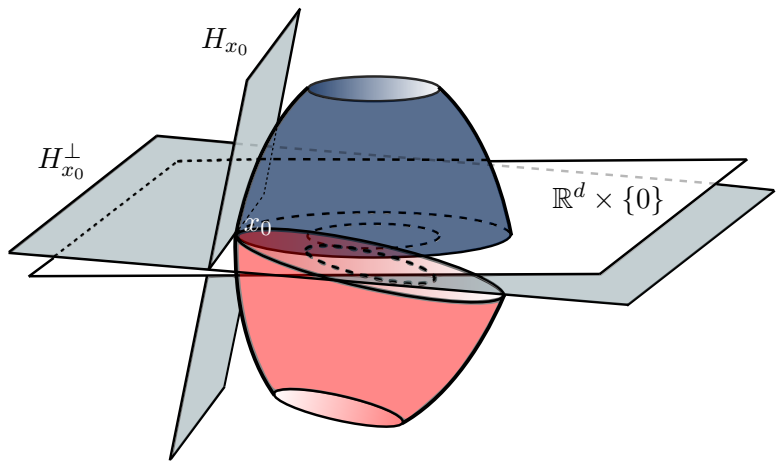}
\caption{\footnotesize{The light red coloured surface
is the reflection of the original surface (coloured in blue) with respect to the hyperplane $H_{x_0}^\bot$,
where $x_0$ is a fixed point on the free boundary. To construct the Blaschke extension,
we first extend $u_*$ for $x_{d+1}<0$ by the lower envelope of its extreme supporting planes at the free boundary points,
and then do the reflection described in the image, for the extended $u_*$.
}}
\label{Fig-Blaschke}
\end{figure}

The projections of the free boundaries of the reflections onto $T_0=\R^d\times\{0\}$ always contain $\fb u$
near the contact point. Thus we can choose 
a dense sequence of points near a fixed point $x_0$ on the free boundary and let $m\to \infty$ to obtain a 
convex body $\mathcal{S}_B$  as the limit of the nested sequence of $\mathcal{S}^m$.
We get that $\Gamma_{u_*} \subset \partial \mathcal{S}_B $ in a neighbourhood of $x_0$. 
We call $\mathcal{S}_B$  the \textbf{Blaschke reflection body}.
Clearly, $\mathcal{S}_B$ defined in this way, depends on $x_0$ and the fixed neighbourhood of it.
However, we will always work locally in a neighbourhood of a given point,
and for this reason there is no need to keep track of the dependence of $\mathcal{S}_B$
on a point $x_0$.

\medskip 

Using $\mathcal S_B$ as a barrier we wish to show that $u_*$ satisfies the free boundary condition
in generalized sense (subsection \ref{subsec-FB}). To begin with, we first look at the regular free boundary points  
where the free boundary condition may fail.
\begin{lem}\label{blya}
Let $x_0\in \fb{u_*}$ be a regular point of the free boundary, and  such  that  the  slope  at $x_0$
satisfies
$$|\nabla u_*(x_0)|<\lambda_0 .$$
Then the boundary of the convex body  $\mathcal{S}_B$ near $x_0$ is a graph in the direction of $x_{d+1}$ axis and has  slope    
strictly less than $\lambda_0$.  Here $\mathcal{S}_B$ is the Blaschke reflection body which contains 
$\mathrm{Hull}(\graph{u_*})=\mathcal S^+_*\cap \{x_{d+1}\ge 0\}$.
\end{lem}

\begin{proof}  Suppose  the claim fails, then  there  are   
points $p_k\in \mathcal S_B\cap \{x_{d+1}<0\}$   with   
support functions at $p_k$ such that slopes $> \lambda_0$.

Recall that $\mathcal S_B=\mathcal S^+_*\cap \bigcap_{x\in \fb {u_*}} S^-_x$ therefore there must be $x_k\in \fb{u_*}$ such that $p_k\in \mathcal S^-_{x_k}$. 
Denote $M_k\in \graph {u_*} $ the pull back of $p_k$. 
By construction $M_k\not \in \fb {u_*}$ because there the slopes are $\le \lambda_0$. Thus at each $M_k\in \graph{u_*}$ 
there is a unique support hyperplane $H_k$ with slope $> \lambda_0$.
Since $M_k\to x_0$ and $x_0$ is a differentiability point of $\Gamma_{ u_*}$
then we get the limit of $H_k$ is an extreme support hyperplane of $\graph{u_*}$ at $x_0$
with slope $\lambda_0$ but this  is impossible due to our assumption.
Now the fact that $\mathcal{S}_B$ is a graph near $x_0$ easily follows from the 
first claim that the slope at $x_0$ is $<\lambda_0 $.
\end{proof}

\begin{cor}\label{pin}
For each point $p=(p_1, \dots, p_{d+1})$ on $\mathcal S_B$ with $p_{d+1}<0$ we have that
near $x_0$ the boundary of $\mathcal{S}_B$ is a graph of some concave function $U_*$ over $\R^d\times \{0\}$ and  
$\det D^2(-U_*)\ge K_0\psi(|\na U_*|)$. 
\end{cor}
\begin{proof}
By Lemma \ref{blya}, the boundary of $\mathcal S_B$ is a graph near $x_0$ over $\R^d\times \{0\}$. Then apply 
the argument in the proof of \eqref{u-n-K0} to complete the proof of this corollary.
\end{proof}

From Corollary \ref{pin} and Lemma \ref{blya} it follows that 
there must be a neighbourhood  $\mathcal N$ of $x_0$  in $\R^{d+1}$  such that  on $(\mathcal N\cap \mathcal S_B)\cap \{x_{d+1}<0\}$ the slopes are $<\lambda_0$. If not then we will get a ruled piece of $\mathcal S_B$ in $\{x_{d+1}<0\}$ where the 
curvature must vanish. 


\smallskip

With these preparations, we are now ready to show that $u_*$ satisfies the free boundary condition.

\begin{lem}\label{Lem-FB-cond-elliptic}
For all regular points $x_0 \in \Gamma_{u_*}$ we have $|\nabla u_*(x_0)|=\lambda_0$. 
\end{lem}

\begin{proof}
Assume for contradiction, that $x_0\in \Gamma_{u_*}$ is a regular point
where $|\nabla u_*(x_0)|<\lambda_0$.
Let $H_0$ be the hyperplane in $\R^{d+1}$ supporting the graph of $u_*$
at $(x_0, 0)$ and having slope $\lambda_0$.
Notice that $H_0$ is the tilted copy of the support hyperplane $H_{x_0}$.
By construction of the extension $\mathcal S_B$, we have that $H_0$ cuts a 
non-trivial cap from $\mathcal S_B$.

We next take $\delta>0$ small and let $H_\delta$ be the parallel translation of $H_0$
by $\delta>0$ towards $\Omega$. For $\delta>0$ small enough we get that
$\mathcal{S}_B$ is seen as a concave function $\widehat U_*$ from  $H_\delta$.
Since the equation is invariant with respect to  rotations and translations in $T_0$, without loss of generality we assume that $x_0=0$ and the inner normal at $x_0$ is in the direction of $x_1$ axis. Then $H_\delta$ can be represented as $x_{d+1}=\lambda_0(x_1-\delta)$. For sufficiently small $\delta$ 
the set $\mathscr D_\delta=\{x\in T_0: U_*>\lambda_0(x_1-\delta)\}$ is a convex domain that can be projected on 
$T_0$ in one-to-one fashion thanks to Corollary \ref{pin}.
%
%
Let $\widehat U_*=u_*-\lambda_0(x_1-\delta)$ and $\widehat  u_\delta$ be the unique solution to 
\begin{equation}\label{tiltup-u}
\left\{
\begin{array}{lll}
\det D^2(-\widehat u_\delta) = K_0 \psi( |\nabla \widehat u_\delta+\lambda_0e_1| ) &\text{ in } \mathscr D_\delta, \\
\widehat u_\delta = 0 \ \ \  &\text{ on }  \p\mathscr D_\delta.
\end{array}
\right.
\end{equation}

We claim that $\widehat  u_\delta< \widehat U_*$. Indeed, to see this it is enough to show that
\begin{equation}\label{a1}
\det D^2 (-\widehat U_*) > K_0 \psi( |\nabla \widehat U_*+\lambda_0e_1| )  \qquad \text{ in } \{\widehat U_*>0\} ,
\end{equation}
as then the desired inequality will follow from the comparison principle.
Observe, that we have a non-strict inequality in \eqref{a1} by construction
(see e.g. the proof of Lemma \ref{Lem-subsol-exists}).
For the strict inequality, we notice that for each $X$ on the graph of $ U_{*}$
there exists (by compactness) a reflected function $U_{*,x}$ (representing $\mathcal S_x^-$) touching the graph of 
$U_{*}$ at $X$  having its pull back on $\mathcal S_B\cap \graph {u_*}$ as interior point (which is possible by construction) 
and staying above $\graph{  U_{*}}$.
Since $\widehat  U_{*, x} :=U_{*, x}-\lambda_0(x_1-\delta)$ solves the PDE in \eqref{a1} with equality, then Hopf's lemma (see Lemma \ref{Lem-smallest-solves-PDE})
implies that the two functions $\widehat  U_{*,x}$ and $\widehat  U_{*}$ must coincide
unless we have a strict inequality in \eqref{a1}.

\smallskip

We thus obtain that $\widehat u_\delta< \widehat U_*$, and to complete the argument it is left to
show that $u_\delta=\widehat u_\delta+\lambda_0(x_1-\delta)$
intersects $\R^d\times \{0\}=\{x_{d+1}=0\}$ by a slope bounded above by $\lambda_0$.
This will then produce a smaller super-solution to \eqref{prob-main}, contradicting the
definition of $u_*$. 
Assume that for any $\delta>0$ small there exists a point $x_\delta \in \Gamma_{u_\delta}$
such that the slope of $u_\delta$ at $x_\delta$ is larger than $\lambda_0$.
Extracting a sequence $\delta_j \to 0$ we get that $x_{\delta_j} \to x_0$,
and $u_{\delta_j} \to U_*$ in the set $\{U_* \leq 0\}$, where $U_*$ is defined (near $x_0$) so that 
$\graph{U_*}=\{-\e<x_{d+1}\}\cap \mathcal S_B$ for  some small 
$\e>0$ depending on $x_0$, see Lemma \ref{pin}.
But then we should have that the support plane of $u_{\delta_j}$ converges to
the support plane of $U_{*}$ at $x_0$ as $j\to \infty$.
However, the slope of (any) limit of these planes has magnitude at least $\lambda_0$,
which by construction cannot be a support plane of the Blaschke extension at $x_0$.
This contradiction completes the proof of the gradient condition on the free boundary.
\end{proof}

\subsection{$C^{\infty}$ regularity of the free boundary}
In this section we complete the proof of Theorem B.
We recapitulate the statement as follows 
\begin{prop}\label{prop-C11}
There is $K_0>0$ small enough depending only on $\lambda_0$, $\Om$, and $\psi$ such that 
for any $0<K<K_0$ the minimal solution $u_*$ to \eqref{prob-main}
with curvature measure $K$ has $C^\infty$ smooth free boundary.
\end{prop}

The idea is to reduce the regularity of the free boundary to interior regularity
by extending a solution from the free boundary below $\R^d\times \{0\}$.
For this we first show that the free boundary is $C^1 $ regular for small
enough $K$, by reducing matters to the case of $K=0$. Relying on $C^1$ regularity
of the free boundary and interior smoothness of solutions we show that
the free boundary rolls freely inside a ball of a uniform radius.
This enables us to apply the same construction, as we used for establishing
the existence of a solution, starting from the convex hull
of the free boundary as the initial domain $\Omega$.

\begin{sublem}\label{sublem-C1}
Assume that $\psi\in C^\infty$ such that $\psi>0$ and 
$\sup_{B(0, \lambda_0)}\psi(\xi)<\infty$. Let $u_K$ be the minimal solution of 
\begin{equation}
\begin{cases}  
\det D^2 (-u_K)=K\psi( | \na u_K |), &\text{$ \{ u_K>0  \} \setminus \overline{\Om}$}, \\ 
 u_K=h_0 ,&\text{$ \partial \Omega    $}, \\
 |\nabla u_K| =\lambda_0, &\text{$\Gamma_{u_K}$}.
 \end{cases}
\end{equation}
corresponding to the given constant $K>0$.
\begin{itemize}
\item [$\bf 1^\circ$]
There is a $K_0$ small such that for any $0<K<K_0$ the normal of the free boundary $\fb{u_K}$ 
is uniformly continuous (consequently $\Gamma_{u_K}$ are uniformly $C^1$).
In other words, 
there is a critical $K_0$ small such that the following holds: 
for every $\e>0$ there is $\delta>0$ such that if 
\[\sup_{\begin{subarray}{c}x_K,  y_K\in \fb {u_K}\\ 
K\in (0, K_0)\end{subarray}}
|x_K-y_K|<\delta
\]
then 
\[\sup_{\begin{subarray}{c}
x_K,  y_K\in \fb {u_K}\\ 
K\in (0, K_0)
\end{subarray}}|\nu(x_K)-\nu(y_K)|<\e.\]
\item [$\bf 2^\circ$]
Moreover $\na u_K|_{\fb{u_k}}=\lambda_0 \nu$ where 
$\nu$ is the unit inner normal of $\fb{u_K}$. 
\end{itemize}
\end{sublem}
\begin{proof}
 $\bf 1^\circ$ \ If the claim fails for some $\e_0>0$ then there are sequences $K_j\downarrow 0$, solutions  $u_j:=u_{K_j}$ to \eqref{prob-main},
 and points $x_j,y_j \in \fb{u_j}$ and such that 
\[|x_j-y_j|<\frac1j,\quad   \hbox{but}\quad |\nu(x_j)-\nu(y_j)|\ge \e_0>0.\]
%
Recall that $u_j$ is the minimal solution defined by 
$u_j=\inf_{\W_j}v$, where
$$
\W_j:=\W_-(K_j, \lambda,\Om)=
\left\{v \ \hbox{concave} : \  \det( D^2 (-v)) \ge K_j \psi(|\na v|), \ \  \left|\na v({\Gamma_{u_K}})\right|\le\lambda_0\right\}.
$$
Note that  \eqref{ineq-star} holds. 
Clearly for $u_{1}$ we have $\dfrac{\det (-D^2 u_1)}{\psi^{}(| \na u_1 | )}=K_1 \ge K_j$ therefore
\[u_1\in \W_j.\]
Consequently $u_j\le u_1$ and 
\[\po {u_j}\subset \po{u_1},\quad  \forall j\ge 1.\]
This provides uniform bound for the convex bodies of $\mathscr K_j=\hbox{Hull}(\graph {u_j})$
and applying Blaschke's selection principle it follows that for some subsequence 
$j_m$ there  is a concave function 
$u_0=\lim_{j_m\to \infty}u_{j_m}$ such that $\det(-D^2 u_0)=0$ in $\{u_0>0\}\setminus\overline{\Om}$.

We claim that $\big| \na {u_0}\vert_{\Gamma_{u_0}} \big|=\lambda_0$. 
If not then there is a point $z_0\in\R^d\setminus \hbox{Hull}(\Gamma_{u_0})$ such that 
$u_0$ at $z_0$ has an extreme  support hyperplane $H$ with slope $\lambda<\lambda_0$. Due to uniform convergence near $\p \Om$ we see  that $H$ must touch $\p\Om$ at some $\xi_0$. 
Consequently we can tilt 
$H$ at $\xi_0$ and obtain another plane $H_\delta$ with larger slope $\lambda+\delta<\lambda_0$, 
again touching $\Om$ at $\xi_0$. Choosing $\delta$ small enough and 
using the construction of \eqref{tiltup-u} with respect to the hyperplane $H_\delta$
we can construct    a solution $\widehat u_\delta$ such that $\sup\widehat u_\delta\le C K_j$
for $K_j$ sufficiently small (this follows from the Aleksandrov maximum principle). Choosing $\delta$ small enough we can construct a super-solution from $\widehat u_\delta$ 
which is below $u_{K_j}$ and consequently we will reach   a contradiction since $u_{K_j}$ is minimal.

Passing to the limit we have that at $x_0=\lim_{j_m}x_{j_m}$ the support 
hyperplanes of codimension 1 at the free boundary of $u_j$ will converge to 
two distinct support hyperplanes with normals $\nu_1(x_0), \nu_2(x_0)$ at $x_0$. We get 
\[|\nu_1(x_0)-\nu_2(x_0)|\ge\e_0\]
which is a contradiction in view of Theorem A. 

\medskip 
$\bf 2^\circ$ Let ${x_i}\subset \po{u}$ such that $x_i\to x_0\in \fb u$. Let 
$\nu_0$ be the outer unit normal at $x_0$, which is unique by the first part $\bf 1^\circ$. 
We can extract a subsequence $i_m$ such that $\{\na u(x_{i_m})\}$ is convergent to some vector $\xi_0$, and the 
gradient estimate forces $|\xi_0|\le \lambda$. Since the 
tangent planes at $x_i$ must converge to an extreme  support plane of $\graph u$ at $x_0$, it follows that 
$\xi_0$ and $\nu_0$ are collinear. Suppose that $|\xi_0|<\lambda_0$ then  
$x_{d+1}=\xi_0\cdot (x-x_0)$ is an extreme  support hyperplane at $x_0$ which is in contradiction with the fact that 
at $x_0$  the free boundary condition is satisfied in the classical sense. Hence $|\xi_0|=\lambda_0$
and the desired result follows.
\end{proof}

\begin{lem}\label{lem-ext-ball}
Let $u$ be as in Proposition \ref{prop-C11}. Then there is a constant $R_0$ depending only on $\lambda_0, d, \psi$ such that 
for every $x_0\in \fb u$ there is $z_0$ such that $x_0\in \p B(z_0, R_0)$ and $\fb u\subset B(z_0, R_0)$. $R_0$ does not depend on $K_0$.
\end{lem}
\begin{proof}
Let the origin  $0\in \Om$ and $v$ be the Legendre transform of $-u$. Then we have that 
$\det D^2v=\frac1{K_0\psi(|z|)}$ in $0<|z|\le \lambda_0$ (cf. Lemma \ref{Lem-smallest-solves-PDE}). 
On the boundary $|z|=\lambda_0$, in view of Sublemma \ref{sublem-C1}, we have that $v(z)=x\cdot  z=\lambda_0\nu\cdot \na v$, $\nu$ is the unit inner normal.
Thus $v$ verifies  the oblique boundary condition $\lambda_0 v_\nu+v=0$ on $|z|=\lambda_0$. Note that $v$ is $C^\infty$ smooth near the sphere $|z|=\lambda_0-\delta$, for sufficiently small $\delta>0$. 
Taking the logarithm $\log \det D^2 v=-\log K_0 -\log \psi(|z|)$ and applying the   $C^{1, 1}$
boundary estimates \cite[pp. 23-29]{Urbas-CVPDE}  we get 
\[\sup_{\p B(0, \lambda_0)}|D^2 v|\le C\]
where $C=C(d, \lambda_0, \|\log\psi\|_{C^{1,1}})$ does not depend on $K_0$.
Pulling back to $-u$ we infer that 
$$
\liminf\limits_{x\to x_0} D^2 (-u(x))\ge c_0 \hbox{Id}, \ \ \forall x\in \po u, \ \ \forall x_0\in \fb u.
$$  for some positive constant $c_0$ and the result follows.
\end{proof}


Consider the problem \eqref{prob-main} with $K_0$ sufficiently small,
such that there exists a solution $u$ with free boundary
rolling freely inside a ball of radius $R_0$.
By Lemma \ref{lem-ext-ball} we have that smallness of $K_0$
is sufficient for this purpose.
Set $\Om_u = \mathrm{Hull}(\fb{u})$, we claim that in a neighbourhood of $\Om_u$ contained in $\R^d\setminus \overline{\Om_u}$ 
the minimal solution $u$ can be extended as a solution.

More precisely, there exists $U:\R^d\setminus \Om_u \to \R$ concave
solving the following problem:
\begin{equation}\label{U-extension}
\begin{cases}  
\det D^2 (-U)=K_0\psi(|\na U|), &\text{  $\left( \R^d \setminus \overline{\Om_u} \right)$}, \\ 
  U=0 ,&\text{  $ \fb u    $}, \\
 |\nabla U| =\lambda_0, &\text{ $\fb u  $}.
 \end{cases}
\end{equation}

To construct such extension, we will treat $\Omega_u$ as the new initial domain $\Omega$.
We next modify slightly the Definition \ref{defn-Aleks}, by defining the set
$\mathbb{W}_-=\W_-(K_0, \lambda_0, \Om_u)$ as the class of all concave functions $u:\R^d \to \R $
satisfying
\begin{equation}\label{super-sol-curv-ext}
\begin{cases}  
\det D^2 (-u)\geq K_0\psi(|\na u|), &\text{$ \{u<0 \}  \setminus \overline{\Omega_u} $}, \\ 
  u\equiv 0 ,&\text{$\overline{\Omega_u}    $}, \\
 |\nabla u|\le \lambda_0, &\text{$\partial \Omega_u $}.
 \end{cases}
\end{equation}

Next, relying on Lemma \ref{lem-ext-ball}, which ensures that $\Omega_u$
rolls freely inside a ball of some radius $R_0$, we invoke the 
Perron method utilized in Section \ref{subsec-existence}
with the gradient condition on the free boundary replaced by gradient
condition on the boundary of the initial domain. 
The details are the same, and we will omit them.

Finally, defining 
\begin{equation}
\widetilde u=
 \begin{cases}  
u&\text{  $ \{u>0\}\setminus \overline{\Om} $}, \\ 
 U ,&\text{  $ \{U<0\}    $}, \\
 \end{cases}
\end{equation}
and using the fact that the gradients of $u$ and $U$ agree on the free boundary $\fb u$
we easily get that $\widetilde{u}$ defines a solution across the free boundary $\fb u$.
Applying the local $C^2$ estimates of Pogorelov \cite{Pog-MA}, \cite{Figalli} to $\widetilde{u}$ the result of Proposition \ref{prop-C11}
follows.

\bigskip

\begin{proof}[Proof of Theorem B]
Putting together Lemmas \ref{Lem-smallest-solves-PDE} and Lemma \ref{Lem-FB-cond-elliptic}
we get that $u_*$ defines a solution to \eqref{prob-main}.
The regularity of the free boundary $\fb{u_*} $is given by Proposition \ref{prop-C11}.
\end{proof}

\appendix

\section{Non existence of $K$-surfaces}

The aim of this appendix, is to show, based on an example of radial solutions,
that depending on the parameters of \eqref{prob-main}, there can be no solution to
\eqref{prob-main}.

\subsection{Radially symmetric solutions}\label{subsec-radial}
In this subsection, using the moving plane method of Aleksandrov,
we show that any solution to \eqref{prob-main} with $K_0>0$ and $\Om$ a ball,
must be radially symmetric. Relying on this, we next show,
by an explicit example, that there might be no solution to \eqref{prob-main}
for general values of parameters.

\begin{lem}\label{lem-Aleksandrov}
Let $u$ be a uniformly convex solution of \eqref{prob-main} such that $\psi(\xi)=\psi(|\xi|)$, $\xi\in \R^d$, 
$\Om$ is a ball centred at the origin 
and let the free boundary $\Gamma_u$ be $ C^2$.
Then $u$ is radially symmetric and the free boundary is sphere of codimension 2.
\end{lem}
\begin{proof}
We will use Aleksandrov's moving plane method \cite{Serrin}. Let $e_1\in\{x_{d+1}=0\},$ be the unit direction of the $x_1$ axis and 
$T$ a hyperplane perpendicular to $e_1$. In other words, $T$ is defined as 
the set of all points $X\in \R^{d+1}$ such that $x_1=t$ where $t\in \R$.
This is a one parameter family of parallel planes moving orthogonal to $e_1$.
Let us move the plane $T$ from $t=-\infty$ until it first time hits  $\fb u$, say at $t=t_1$.
Then for each $t>t_1$ we let 
$$D_t=\bigcup_{s\in(t_1, t)}T_t\cap\{u>0\}$$ 
and $D^*_t$ be the symmetric domain  of $D_t$ with respect to $T_t$. 
Similarly, we let $S_t$ to be the surface that $T$ cuts from $\graph u$ and $S^*_t$ its symmetric reflection with respect to $T_t$.

We consider several cases:

{\bf Case 1:}
Let $t_0>t_1$ be such that $S_t^\ast$ and $\graph u$ touch first time at some interior point 
$x_0\in \{u>0\}\setminus\overline{\Om}$ then letting $u^*(x)=u(x^*)$ where $x^*$ is the symmetric point 
determined from the equality
$$x_1^\ast =2t-x_1,\quad  x_i^\ast=x_i, i=2, \dots, d.$$
If we denote $(x_1, x_2, \dots, x_d)=(x_1, y)$ then 
$u^\ast (x^\ast)= u(2t-x_1, y)$. 
Clearly, 
\begin{eqnarray*}
\det D^{2}(-u^*(x))&=&(\det D^2(-u))(2t-x_1, y) \\
&=&\psi(\na u^\ast)\\
&=&\psi((\na u)(2t- x_1, y))
\end{eqnarray*} 
and consequently we have that 
\begin{eqnarray}
\log\psi(\na u^*)-\log\psi(\na u)&=&\log \det D^{2}(-u^*(x))-\log\det D^{2}(-u(x))\\
&=&a_{ij}[u_{ij}^*(x)-u_{ij}(x)]
\end{eqnarray}
where 
\begin{eqnarray*}
a_{ij}(x)&=&\int_0^1\frac{\p\det(-D^{2}(su^*(x)+(1-s)u(x)))}{\p r_{ij}}ds\\
&=&
\int_0^1\left[{\hbox{cof}\left(-D^{2}\left(su^*(x)+(1-s)u(x)\right)\right)_{ij}}\ \right]ds.
\end{eqnarray*}
Here $\hbox{cof}$ is the cofactor matrix.
  By assumption $u$ is uniformly convex, i.e. 
  $-D^2u(x)\ge c\hbox{Id}, c>0, x\in \{u>0\} \setminus \overline{\Om}$ thus $a_{ij}$ is uniformly elliptic matrix
  with ellipticity constants depending only on $c$ and $C^{2, \alpha}$ norm of $u$ (note that if $\fb u\in C^{2, \alpha}$ then from existing theory of $K$-surfaces we can conclude that 
  $\graph u\in C^{2,\alpha}(\overline{\{u>0\}\setminus\Om})$). 
 Similar argument implies that 
  \[
  \log\psi(\na u^\ast)-\log\psi(\na u )=b\cdot(\na u^\ast-\na u), \quad b=\int_0^1\frac{\na\psi(s\na u^\ast+(1-s)\na u)}{\psi(s\na u^\ast+(1-s)\na u)}ds.
  \]
At the touching point we can apply Hopf's maximum principle to the linearised equation 
$a_{ij}w_{ij}-b_iw_i=0, w=u^\ast-u$ to conclude that $u^*=u$ in 
$D^*_t$. Hence if we have interior touch between $S_t^*$ and $\graph u$ then 
$\graph u$ must be symmetric w.r.t. $T_{t_0}$.

\medskip

{\bf Case 2:} Suppose that the first touch of $S_t^*$ and $\graph u$ happens at the boundary 
$x_0\in \fb u\cap \p D^*_t$. Since we are assuming that the free boundary 
is  $C^{2, \alpha}$ smooth then at $x_0$ we have that $|\na u(x_0)|=|\na u^*(x_0)|=\lambda_0$.
We can apply Hopf's lemma at $x_0$ to infer that $\graph u$ is symmetric w.r.t. $T_{t_0}$

\medskip

{\bf Case 3:} Finally let us consider the case when $x_0\in T_{t_0}\cap \fb u$, where 
$t_0$ is the value for which $S_t^*$ and $\graph u$ touch first time.
For this case we need to apply a well known computation of Serrin \cite{Serrin}.
We consider an orthonormal frame at $x_0$, the $x_n$ 
axis being directed along the inward normal to $\fb u$. Assuming that 
near $x_0$ the free boundary has the representation 
\[x_n=\phi(x_1, \dots, x_{n-1})\]
for a convex function $\phi\in C^{2,\alpha}$.
Using Serrin's computation \cite[pp. 307-308]{Serrin}  we have 
\[
\sbox0{$\begin{matrix}{}&{}&{}\\{}& -\lambda \phi_{ij}&{}\\ {}&{}&{}\end{matrix}$}
D^2 u(x_0)=\left[
\begin{array}{c|c}
\usebox{0}
&\makebox[\wd0]{\large $0$}\\
\hline
  \vphantom{\usebox{0}}\makebox[\wd0]{\large $0$}&\makebox[\wd0]{\large $u_{nn}$}
\end{array}
\right]
\]
From the equation $\det (-D^2 u)=K\psi(\na u)$ we find that 
$$-u_{nn}=\frac{\psi(\lambda)}{\lambda^{n-1}\det\phi_{ij}}.$$
Hence at $x_0$ $u$ and $u^*$ have identical gradient and Hessian. 
Applying the Lemma 2 from \cite{Serrin} we again conclude that 
 $\graph u$ is symmetric w.r.t. the plane $T_{t_0}$.

Summarizing, we see that in any case for each unit vector $e$ there is a position for the plane 
$T_t$ such that $\graph u$ is symmetric w.r.t. $T_{t_0}$. Suppose that 
$T_{t_0}$ (the one where the first touch happens for given unit vector $e$)
does not intersect $\Omega$. From above analysis it follows that 
$S_{t_0}^*$ must contain horizontal hole $\Omega\times\{1\}$.
But $S_{t_0}$ and $S_{t_0}^*$ are symmetric hence both must contain the horizontal 
holes. Thus $T_{t_0}$ must intersect $\Om=B(0, 1)$. 
Since $\Om$ is symmetric  w.r.t. 0 then it follows that $T_{t_0}$ passes through the 
origin. Since $e$ is arbitrary then it follows that $u$ is radially symmetric. 
\end{proof}

\subsection{Explicit computations}\label{subsec-comp}
Consider (\ref{prob-main}) with $\Omega= B(0,R_0)$ and $g\equiv h_0>0$ on the boundary of $\Omega$.
We search for a radial solution to (\ref{prob-main}) with given $K=K_0 >0$, $\lambda=\lambda_0$ under these conditions.
Let $u$ be radial, i.e. there exists a function $v:[R_0,\infty) \to \R_+$ such that $u(x) = v(r)$ for $|x|=r$.
Then, a direct computation reveals
\begin{equation}\label{u-partials}
u_{ii }(x)  =\frac{x_i^2}{r^2} \left( \ddot{v} - \frac{\dot{v}}{r}  \right) + \frac{\dot v}{r}
\qquad \text{and}  \qquad 
u_{ij} (x) = \frac{x_i x_j }{r^2} \left( \ddot{v} - \frac{\dot{v}}{r}  \right),
\end{equation}
where $1\leq i \neq j\leq d$.
To compute the determinant of the Hessian of $u$ we will rely on \emph{matrix-determinant lemma},
which states that for invertible $A\in M_d(\R) $ and column-vectors $u,v\in \R^d$  one has
$ \mathrm{det} (A  + u v^T)  = ( 1+v^T A^{-1} u ) \mathrm{det} A $. To apply this identity,
set $ \ddot{v} - \frac{\dot v }{ r } := \alpha $, $\frac{\dot{v} }{r} : = \beta$, and $u^T : = \frac 1r (x_1,...,x_d )$;
we get
$$
\det D^2 u = \mathrm{det} ( \alpha u u^T + \beta \mathrm{Id} ) : = \mathcal{D} .
$$ 
We will first assume that both $\alpha$ and $\beta $ are non-zero,
in which case we obtain
\begin{multline}\label{det-comput}
\mathcal{D} = \alpha^d \mathrm{det} \left( u u^T + \frac{\beta}{\alpha} \mathrm{Id} \right) =  \ \text{(matrix-determinant lemma)} \\
\alpha^d \left( 1+u^T \hspace{0.05cm} \frac{\alpha}{\beta} \mathrm{Id} \hspace{0.05cm}  u \right) \mathrm{det} \left( \frac{\beta}{\alpha} \mathrm{Id} \right) =
\alpha^d \left(  1+ \frac{\alpha}{\beta} |u|^2  \right) \left( \frac{\beta}{\alpha} \right)^d = 
\beta^d + \alpha \beta^{d-1},
\end{multline}
where the last equality is in view of the identity $|u|=1$. We now observe that the formula \eqref{det-comput}
remains valid when either of $\alpha$ or $\beta$ becomes 0, we thus have \eqref{det-comput} for all $\alpha, \beta$.
We now plug the values of $\alpha$, $\beta$ into the formula for \eqref{det-comput}, and obtain
$\det D^2 u  = \ddot{v} \left(  \frac{\dot v }{r}  \right)^{d-1}$. Finally, getting back to \eqref{prob-main},
the equation becomes
\begin{equation}\label{ODE-curv}
 \ddot{v} \left(  \frac{\dot v }{r}  \right)^{d-1} = (-1)^d K_0.
\end{equation}
From this ODE we get
\begin{equation}\label{ode-for-v}
(\dot{v}(r) )^d  = (-1)^d K_0 r^d + A,
\end{equation}
where $A$ is some constant. One has two cases here.

\textbf{Case 1.} Take $A=0$. Then $u(x) = B - K_0^{1/d} \frac{|x|^2}{2}$ for a constant $B$.
From the boundary condition on $u$ we get
\begin{equation}\label{par-A-0}
u(x)  = h_0 +  \frac{K_0^{1/d}}{2} (R_0^2 -|x|^2). 
\end{equation}
Finally, the free boundary condition is satisfied, if and only if
\begin{equation}\label{res1}
\lambda_0 = K_0^{1/d} \left( R_0^2 + 2 h_0 K_0^{-1/d}  \right)^{1/2}.
\end{equation}
We conclude that under condition \eqref{res1} the paraboloid given by 
\eqref{par-A-0} produces a solution to \eqref{prob-main}.

\textbf{Case 2. } Assume $A \neq 0$. Then, according to Chebichev's 
theorem on the integration of binomial differentials (\ref{ode-for-v}) can be solved explicitly (in terms of elementary functions) if and only if $d=2$.
Taking $d=2$ and solving \eqref{ode-for-v} we get
$$
v(r) = B - \frac{1}{2} r \sqrt{K_0 r^2 + A} -\frac{A K_0^{-1/2}}{2}
\ln |  r+ K_0^{-1/2} \sqrt{r^2 K_0 +A} | =: B - \varphi_A(r),
$$
where $A,B$ are constants, $A\neq 0$ and
\begin{equation}\label{const-A-cond}
A \geq -R_0^2 K_0.
\end{equation}
From the boundary condition on $u$, we get
$$
B = h_ 0 + \varphi_A(R_0),
$$
and hence 
$$
 v(r) = h_0 + \varphi_A (R_0) - \varphi_A(r), \qquad r \geq R_0. 
$$
From the free boundary condition, we see that $A= \lambda_0^2 - K_0 r^2$,
if $v(r)=0$. Comparing this expression for $A$ with \eqref{const-A-cond}
we get
$$
\lambda_0^2 - K_0 r^2 \geq -K_0 R_0^2 ,
$$ 
and finally
\begin{equation}
r \leq \left( R_0^2 + \frac{\lambda_0^2}{K_0} \right)^{1/2} =: R_{max},
\end{equation}
which shows that radius of the 0-level set of $v$ must be bounded above by $R_{max}$.
Thus, taking into account the monotonicity of $v$, we conclude that $v$ gives
a solution if and only if $v(R_{max}) \leq 0$. 
Plugging the value of $A$ into $\varphi$ we get
$$
\varphi_A(x)  = \frac 12 r \lambda_0 + \frac{ ( \lambda_0^2 - K_0 r^2 ) }{2} K_0^{-1/2}
\ln | r+ K_0^{-1/2} \lambda_0 | = : \varphi(r).
$$
Hence, the function $u(x) = h_0 + \varphi(R_0) - \varphi(|x|)$ produces a solution
to \eqref{prob-main} if and only if
\begin{equation}
h_0 + \varphi(R_0) - \varphi( R_{max} )  \leq 0.
\end{equation}
Clearly the validity of this condition depends on parameters $h_0, K_0, \lambda_0$,
and is not always satisfied.

\end{document}